\documentclass[11pt]{article}
\usepackage[T1]{fontenc}
\usepackage{lmodern}
\usepackage{amsmath,amssymb,amsthm,mathtools}
\usepackage[a4paper,margin=27mm]{geometry}
\usepackage[hidelinks]{hyperref}
\usepackage{mathrsfs}
\DeclareMathOperator{\di}{di}
\newcommand{\DSdual}{\mathsf{DS}^{*}}
\usepackage{tikz}
\usepackage{booktabs}

\tikzset{
  rational dyck grid/.style={draw=black!35,line width=0.25pt},
  rational dyck diagonal/.style={draw=black,line width=0.4pt},
  rational dyck path/.style={
    draw=red,line width=1.1pt,line cap=round,line join=round}
}

\newcommand{\RationalDyckDiagram}[4][3.2mm]{%
  \begin{tikzpicture}[
    x=#1,y=#1,baseline=(current bounding box.center)]
    \pgfmathtruncatemacro{\RFMlastcol}{#3-1}
    \ifnum\RFMlastcol>0\relax
      \foreach \RFMcol in {1,...,\RFMlastcol}{%
        \pgfmathtruncatemacro{\RFMheight}{floor(#2*\RFMcol/#3)}
        \ifnum\RFMheight>0\relax
          \foreach \RFMrow in {1,...,\RFMheight}{%
            \draw[rational dyck grid]
              (\RFMcol,{\RFMrow-1}) rectangle ++(1,1);
          }
        \fi
      }
    \fi
    \draw[rational dyck grid] (0,0)--(#3,0)--(#3,#2);
    \draw[rational dyck diagonal] (0,0)--(#3,#2);
    \draw[rational dyck path] (0,0)
      foreach \RFMrise in {#4}{-- ++(1,0)-- ++(0,\RFMrise)};
  \end{tikzpicture}%
}

\newcommand{\RationalDyckCell}[3]{%
  \begin{minipage}[c]{0.24\textwidth}
    \centering
    \((#3)\)\par\smallskip
    \RationalDyckDiagram{#1}{#2}{#3}\par
  \end{minipage}%
}

\DeclareMathOperator{\area}{area}
\DeclareMathOperator{\dinv}{dinv}

\DeclareMathOperator{\defc}{defc}

\theoremstyle{definition}
\newtheorem{definition}{Definition}[section]
\newtheorem{example}[definition]{Example}
\theoremstyle{plain}
\newtheorem{lemma}[definition]{Lemma}
\newtheorem{corollary}[definition]{Corollary}
\newtheorem{theorem}[definition]{Theorem}

\theoremstyle{remark}
\newtheorem{remark}{Remark}
\title{An involution on Dyck paths in the region $\defc\le\min(\area,\dinv)$ that interchanges area and dinv}
\author{Graham Hawkes\thanks{\texttt{grhmhwks@gmail.com}}}
\date{}

\begin{document}

\maketitle

\begin{abstract}
We construct an explicit involution on Dyck paths satisfying $\defc\le\min(\area,\dinv)$ that interchanges area and dinv, where $\defc=\binom n2-\area-\dinv$ and $n$ is the semilength.
In addition, we give an involution on Dyck paths with $\defc \le 2n-8$ that interchanges area and dinv.   Finally, we give an explicit partition formula for the portion of the $q,t$-Catalan polynomial of total degree at least $\binom n2-2n+8$, thereby proving a conjecture of Lee and Li~\cite[Conjecture~4]{LeeLi11}. 
Altogether our results give a combinatorial explanation of $q,t$-Catalan symmetry in the region $\defc\le\max(2n-8,\min(\area,\dinv))$. Our constructions rely on a number of new combinatorial objects developed here
and three main external tools: the dual Dyck insertion of \cite{Hawkes26}, the Garsia--Milne involution principle~\cite{GarsiaMilne81,Doyle19}, and the partition bijection of ~\cite{LoehrWarrington09}. We also explore rational analogs of these results, give an answer to Question~4.9 of Kiem and Lee~\cite{KiemLee24} and consider deficit up to $3n-15$ briefly. 
\end{abstract}

\section*{Introduction}
For a Dyck path $D$ of semilength $n$, put $M=\binom n2$ and
$\defc(D)=M-\area(D)-\dinv(D)$. We construct an explicit involution on the paths satisfying
\[
 \defc(D)\le\area(D)\le\binom n2-2\defc(D)
\]
that interchanges area and dinv, as well as one with the same properties on the paths satisfying

\[
 \defc(D)\le 2n-8.
\]

In Section~1 we introduce lowers and uppers and show that Dyck paths are in $\area$/$\dinv$ preserving bijection with lowers and also with the union of uppers and full Dyck skeletons.
We argue the desired numerical equality of Dyck paths in the ``flat middle'' region ($\area\in[\defc,M-2\defc]$) follows from showing lower and upper balance here.
In Section~2 we make two main reductions to the outstanding problem and a simple normalizing translation.
In Section~3 we leverage a specific non-existence statement (proved in the appendix) and prove balance by giving a telescoping argument and introducing $pTs$ objects and an involution on them.
In Section~4 we build an explicit bijection out of our previous constructions. In Section~5 we use this bijection along with results of Loehr and Warrington~\cite{LoehrWarrington09} to give a complete combinatorial proof of $q,t$-Catalan symmetry for deficit at most $2n-8$. We conclude the main portion of the paper with an explicit partition formula for this portion of the polynomial in Corollary~\ref{cor:form}, proving a conjecture of Lee and Li~\cite[Conjecture~4]{LeeLi11}. In Section~6 we give the rational analog to Theorem~\ref{thm:flat-middle}, prove $e$-positivity of dual Dyck symmetric functions (and in doing so give an answer to Question~4.9 of Kiem and Lee~\cite{KiemLee24}), and give a string decomposition formula for the $q,t$-Catalan polynomial down to degree $\binom n2-3n+15$.

This work continues the author's efforts to give combinatorial explanations of $q,t$-Catalan symmetry~\cite{Hawkes24,Hawkes26}.
It also follows the work of Lee, Li, and Loehr~\cite{LeeLiLoehr18} and its development by Han, Lee, Li, and Loehr~\cite{HanLeeLiLoehr20}, which approach this symmetry through chain decompositions.

\section{Lowers and uppers}
\subsection{Definitions}

\begin{definition}
An \emph{affine Dyck sequence} is a sequence of integers $d_1,\ldots,d_n$ with $d_{i+1}\le d_i+1$.
\end{definition}

\begin{definition}
A \emph{reverse affine Dyck sequence} is a sequence of integers $d_1,\ldots,d_n$ with $d_i\le d_{i+1}+1$.
\end{definition}

\begin{definition}
A \emph{Dyck sequence} is an affine Dyck sequence with nonnegative entries starting at $0$.
\end{definition}

\begin{definition}
For an integer sequence $C=(c_1,\ldots,c_m)$, its \emph{area}, \emph{dinv}, and \emph{deficit} are
\[
 \begin{gathered}
 \area(C)=\sum_{i=1}^{m}c_i,
 \qquad
 \dinv(C)=\#\{(i,j):1\le i<j\le m,\ c_i-c_j\in\{0,1\}\},\\
 \defc(C)=\binom{m}{2}-\area(C)-\dinv(C).
 \end{gathered}
\]
\end{definition}

\begin{definition}
An \emph{extractable element} is an element $e$ of value $v$ in a Dyck sequence such that $e$ is the leftmost occurrence of $v$, exactly one element of value $v-1$ occurs to the left of $e$, and the element immediately to the right of $e$, if it exists, does not have value $v+1$.
\end{definition}

\begin{definition}
A \emph{Dyck skeleton} is a Dyck sequence with no extractable element except possibly its rightmost element.
\end{definition}

\begin{definition}
A \emph{full Dyck skeleton} is a Dyck sequence with no extractable elements.
\end{definition}

\begin{definition}
A \emph{lower} is a triple $(P,A,B)$ where $|P|+|A|+|B|=n$, $P$ is a Dyck skeleton, $A$ is a nonnegative affine Dyck sequence, $B$ is a nonnegative reverse affine Dyck sequence, $P[-1]:A$ is affine, $B:(\max(P)-1)$ is reverse affine, and $|A|=k+1$, $|B|=k$ for some $k\ge0$.
Its area and dinv are defined by
\[
 \area(P,A,B)=\area(P:A:B)+k,
 \qquad
 \dinv(P,A,B)=\dinv(P:A:B)-k.
\]
\end{definition}

\begin{definition}
An \emph{upper} is a triple $(P,A,B)$ where $|P|+|A|+|B|=n$, $P$ is a Dyck skeleton, $A$ is a nonnegative affine Dyck sequence, $B$ is a nonnegative reverse affine Dyck sequence, $P[-1]:A$ is affine, $B:(\max(P)-1)$ is reverse affine, and $|A|=k$, $|B|=k+1$ for some $k\ge0$.
Its area and dinv are defined by
\[
 \area(P,A,B)=\area(P:A:B)+(k+1),
 \qquad
 \dinv(P,A,B)=\dinv(P:A:B)-(k+1).
\]
\end{definition}

\begin{definition}
\emph{Extraction} removes the leftmost extractable element from a Dyck sequence.
\end{definition}

\begin{definition}
\emph{Injection} of an element of value $v$ into a Dyck sequence inserts it immediately after the leftmost occurrence of $v-1$.
\end{definition}

\begin{lemma}
\label{lem:full-skeleton-statistics}
If $D$ is a full Dyck skeleton of length $n$, then each entry of value $j$ in $D$ has at least two entries of value $i$ to its left for every $0\le i<j$. Moreover,
\[
 \dinv(D)+2\area(D)\le\binom{n}{2}.
\]
\end{lemma}

\begin{proof}
First, we show that each entry of value $g>0$ has at least two $g-1$ entries to its left. If $g$ had only one $g-1$ to its left, it would be extractable unless a $g+1$ were immediately to its right. Consider the run $g,g+1,\ldots,g+t$ starting at this entry, with $t$ maximal. Then $g+t$ would be extractable, a contradiction.

Thus, for any entry of value $j>0$, there are two $j-1$ entries to its left. To the left of the leftmost of those $j-1$ entries there are two $j-2$ entries, and so on down to value $0$. Hence there are at least two entries of every value $0\le i<j$ to the left of $j$, proving the first part.

To count $\dinv(D)+2\area(D)$, declare the owner of a dinv pair to be its rightmost element. For each index $i$, the entry $D_i$ owns at most $(i-1)-2D_i$ dinv pairs by the first part of the lemma, so its contribution to $\dinv(D)+2\area(D)$ is at most $i-1$. Summing over $i$ gives the result.
\end{proof}

\begin{lemma}
\label{lem:extraction}
Suppose that the element $e$ is extracted from a Dyck sequence and then the element $f$ is extracted.  Then $(e,f)$ is a positive reverse affine Dyck sequence.  Further, when $f$ is extracted, it has no more elements to its right than $e$ had when extracted. 
\end{lemma}

\begin{proof}
Before either extraction, $f$ lay either to the left of $e$ or to its right. In the first case, if $f$ lay more than one position to the left of $e$, it would have been extracted first. Since $e$ is extractable, $f=e-1$. In the second case, elements of every value $0,1,\ldots,e-1$ still lie to the left of $f$ before its extraction, so $f\ge e$. The second part follows from the configurations of the two possibilities. The positivity follows from the fact $0$ cannot be extracted.
\end{proof}

\begin{lemma}
\label{lem:injection}
If $1 \leq f\le\max(D)+1$ for a Dyck sequence $D$, then $f$ successfully injects into $D$ and becomes extractable.
\end{lemma}

\begin{proof}
The $f$ is inserted after the first $f-1$ so is extractable unless an element greater than it is immediately to its right, which is impossible since that element was previously immediately to the right of $f-1$.
\end{proof}

\begin{lemma}\label{lem:skeljection}
If $1 \leq f\le\max(D)+1$ for a Dyck skeleton $D$, then $f$ successfully injects and becomes the new leftmost extractable element.
\end{lemma}

\begin{proof}
The only thing to show is that there is no extractable element to the left of $f$.  The element immediately to the left of $f$ is $f-1$ so is not extractable; any element further left cannot be extractable since $D$ was a Dyck skeleton.
\end{proof}

\begin{lemma}
\label{lem:successive-injections}
Suppose $(e,f)$ is a positive reverse affine Dyck sequence and $f$ successfully injects and becomes the leftmost extractable element. Then $e$ successfully injects and becomes the new leftmost extractable element.
\end{lemma}

\begin{proof}
If $e$ is injected to the left of $f$, the element immediately to its left is $e-1$ so not extractable and any extractable element further to the left would contradict $f$ having been the leftmost extractable element. Otherwise, $e=f+1$ and is injected immediately to the right of $f$. The only thing to check is that $f$ is no longer extractable, which follows because $e>f$ and $e$ is immediately to its right.
\end{proof}

\subsection{Bijections}

\paragraph{The bijection $\phi$ from Dyck sequences to lowers.}
Given a Dyck sequence $D$ of length $n$, successively extract elements as long as the $i$th extraction can remove an element outside the last $i+1$ entries of the current sequence. Let $k$ be the number of extractions performed.

Form $B$ from the extracted positive reverse affine Dyck sequence (the sequence formed by appending new extractions to the right) by decreasing each entry by $1$. Let $P$ be the first $n-(2k+1)$ entries of the remaining sequence and $A$ its last $k+1$ entries. Set $\phi(D)=(P,A,B)$.

The inverse $\phi^{-1}$ is given by adding $1$ to each entry of $B$ and injecting the resulting elements (right to left).

\paragraph{The bijection $\psi$ from Dyck sequences that are not full Dyck skeletons to uppers.}

Given a Dyck sequence $D$ of length $n$, successively extract elements as long as the $i$th extraction can remove an element outside the last $i-1$ entries of the current sequence. Let $k+1$ be the number of extractions performed.

Let $B$ be the extracted positive reverse affine Dyck sequence (formed left to right) with each entry decreased by $1$. Let $P$ be the first $n-(2k+1)$ entries of the remaining sequence and $A$ its last $k$ entries. Set $\psi(D)=(P,A,B)$.

The inverse $\psi^{-1}$ is given by adding $1$ to each entry of $B$ and injecting the resulting elements (right to left).

\begin{proof}[Well-definedness of the forward maps]
The extracted elements are all at least $1$, so $B$ has nonnegative entries. The sequence $P$ is a Dyck skeleton because otherwise the process would have continued. The sequence $P[-1]:A$ is affine because $P:A$ is. When the last element was extracted, the element immediately to its left remained in $P$ and was one less than the extracted element and so equal to $B[-1]$. Thus, $B[-1]\le\max(P)$. Hence $B:(\max(P)-1)$ is reverse affine, since $B$ is reverse affine.
\end{proof}

\begin{proof}[Well-definedness of the inverse maps]
The inequality $B[-1]+1\le\max(P)+1$, together with Lemma~\ref{lem:injection} and the fact that $B$ is reverse affine, implies that all injections into $P:A$ succeed.
\end{proof}

Both $\phi^{-1}\circ\phi$ and $\psi^{-1}\circ\psi$ are the identity. This follows immediately from the fact that adding $1$ after subtracting it does not change a value, and that injecting a value after extracting it does not change a sequence.

Conversely, both $\phi\circ\phi^{-1}$ and $\psi\circ\psi^{-1}$ are the identity. Since $B[-1]:(\max(P)-1)$ is reverse affine, the first element is successfully injected into $P:A$ and becomes extractable. Again, because $B[-1]:(\max(P)-1)$ is reverse affine it is inserted to the left of $A$; thus it can be considered as having been injected into $P$ and since $P$ is a Dyck skeleton, the first injected element becomes the leftmost extractable element of $P$ by Lemma~\ref{lem:skeljection} (and thus of $P:A$). Further, each successive injected element becomes the new leftmost extractable element by Lemma~\ref{lem:successive-injections}. Thus, successive extractions undo the injections. Since the last extraction did not occur in the last $|A|$ elements none of them did by Lemma~\ref{lem:extraction} so the process did not terminate before $|B|$ extractions.  After all these extractions, the remaining sequence is $P:A$, with $P$ a Dyck skeleton. The only available extraction would therefore be from $A$ or from the rightmost element of $P$, neither of which is allowed so the procedure stops and returns $(P,A,B)$.

\begin{proof}[Preservation of area and dinv]
First collect the extracted elements to the left of the sequence and consider dinv of the concatenation of these extracted elements and the remaining sequence. Each extraction creates one dinv pair, giving $k$ additional pairs for lowers and $k+1$ for uppers but these are accounted for in the lower or upper statistics. Then moving the extracted sequence to the right of the remaining entries (where it ends up as $B$) and decreasing each of its entries by $1$ preserves the internal dinv of both the extracted and remaining sequences and only potentially changes the dinv between the two sets. For any dinv pair between the two sets, each of the form $i:i$ becomes an $i:(i-1)$ pair under the transformation, and each $i:(i-1)$ pair becomes an $(i-1):(i-1)$ pair. No other pairs appear. Thus, both bijections preserve dinv.  The area is preserved because $k$ or $k+1$ is added to the lower or upper statistic to counteract the lost area. 
\end{proof}

By Lemma~\ref{lem:full-skeleton-statistics}, every full Dyck skeleton $D$ satisfies $\area(D)\le\defc(D)$. Consequently, for fixed deficit, showing that the number of lowers of area $a$ equals the number of uppers of area $a+1$ for $a\in[\defc,M-2\defc)$ shows that the number of Dyck paths is the same at every area in this range.  This will be the subject of the next two sections.

\section{Reduction}
In this section, we make two reductions and a translation, leaving a simpler problem for Section~3. 

Let a ``key'' be a triple $P,X,h$ where $P$ is a Dyck skeleton, $X$ is a non-negative multiset of odd length such that $|P|+|X|=n$ and $h \geq 0$.  A lower or upper $(P',A,B)$ is
associated to the key $(P,X,h)$ if and only if $P'=P$, the multiset union of $A$ and $B$ is $X$, and $\dinv(A:B)=h$. Under this association the set of keys partitions the sets
of lowers and uppers.  To show that there are the same number of lowers of deficit $d$ and area $a$ as uppers of deficit $d$ and area $a+1$ it suffices to show that for each 
key $(P,X,h)$ the numbers of associated lowers of deficit $d$ and area $a$ and associated uppers of deficit $d$ and area $a+1$ are the same. Moreover unless

\[
 \begin{aligned}
 a&=\sum P+\sum X+k,\\
 d&=M-\sum P-\sum X-\dinv(P)-C(P,X)-h,
 \end{aligned}
\]

where $C(U,V)$ is the number of pairs $(u,v)$ with $u\in U$, $v\in V$, and $u-v\in\{0,1\}$ for any multisets of integers $U,V$, both of these numbers are $0$.  Therefore, substituting these values into $a \in [d, M-2d)$ and
rewriting as two inequalities and noting that $M=\binom{|P|+|X|}{2}$, it is enough to show that when

\begin{align}
 h&\ge \mathcal A(P,X):=
 \binom{|P|+|X|}{2}-2\sum P-2\sum X-\dinv(P)-C(P,X)-k,\label{eq:weak-wall}\\
 2h&>\mathcal B(P,X):=
 \binom{|P|+|X|}{2}-\sum P-\sum X-2\dinv(P)-2C(P,X)+k,\label{eq:strict-wall}
\end{align}
then the numbers of lowers and uppers associated to the key $(P,X,h)$ are the same.  We call this showing ``balance'' on the key $(P,X,h)$.

\subsection{Reduction of bounds}

\begin{lemma}
\label{lem:reduction-bounds}
If $(P,X,h)$ satisfies $h\ge\mathcal A(P,X)$ and $X$ supports an upper or lower, then either
\[
 \max(P)=P[-1]+1\quad\text{and}\quad\max(X)\le P[-1]+1,
\]
or $\max(P)=P[-1]$.
\end{lemma}

\begin{proof}
Suppose we have a valid upper or lower $(A,B)$ on $(P,X,h)$. Put $m=\max(P)$. Suppose, for contradiction, that either $P[-1] \leq m-2$ or else that $P[-1]=m-1$ and $\max(X)>m$.
We wish to show $\dinv(A:B) < \mathcal A(P,X)$ meaning 
\[
 \begin{aligned}
 &\binom{|P|}{2}+|P||X|+\binom{|X|}{2}-\dinv(A:B)-\dinv(P) -C(P,X)-2\sum P-2\sum X>k.
 \end{aligned}
\]
Since $P[-1]<\max(P)$, the entry $P[-1]$ is not extractable, so $P$ is a full Dyck skeleton. By Lemma~\ref{lem:full-skeleton-statistics}, it suffices to show
\[
 |P||X|+\binom{|X|}{2}-\dinv(A:B)-C(P,X)-2\sum X>k.
\]

In other words, the number of pairs between $P$ and $X$ and within $X$ that do not form dinv pairs, minus twice the sum of the entries of $X$, must be greater than $k$.

Assign an owner to each such pair as follows:
\begin{enumerate}
 \item[(i)] If one element is in $P$, the owner is the element in $A:B$.
 \item[(ii)] If both elements are in $A:B$, the owner is the larger element, unless the larger element is the leftmost $m+1$ in $A:B$ and the other element is $m-1$; in that case, the owner is the $m-1$.
\end{enumerate}

For each entry of $A:B$, we show that it owns at least one more pair not counted by dinv than twice its value, unless it is the first $m+1$.

\begin{enumerate}
\renewcommand{\labelenumi}{\textbf{(\alph{enumi})}}
\item For $e>0$, an entry $m+e$ (that is not the first $m+1$) owns at least $2e-1$ pairs within $A:B$, since $(m-1):A$ and $B:(m-1)$ are affine and reverse affine, respectively. It also owns pairs with at least $2m+2$ entries of $P$: $P[-1]$, the first $m$ in $P$, and at least two entries of each value $0,1,\ldots,m-1$ to the left of that first $m$. The total is therefore
\[
 2e-1+2m+2=2(m+e)+1.
\]
The first $m+1$ owns at least $2m+2=2(m+1)$ pairs from $P$.

\item Each $m$ in $A:B$ owns $2m$ pairs with the at least two entries of each value $0,1,\ldots,m-1$ to the left of the first $m$ in $P$, and also owns the pair with $P[-1]$, giving $2m+1$ pairs.

\item Each entry $t<m-1$ owns $2t$ pairs with the at least two entries of each value $0,1,\ldots,t-1$ in $P$, and one pair with the first $m$ in $P$, giving $2t+1$ pairs.

\item Each $m-1$ owns at least $2(m-1)$ pairs with the two entries of each value $0,1,\ldots,m-2$ in $P$, plus one additional pair: with $P[-1]$ if $P[-1]<m-1$, or with the first $m+1$ in $A:B$ if $P[-1]=m-1$. So the total is $2(m-1)+1$.
\end{enumerate}

Therefore, the difference is at least $2k+1$ if there is no $m+1$, or at least $2k$ if there is an $m+1$. If there is an $m+1$, then $k\ge1$, so in either case the difference is greater than $k$.
\end{proof}

\subsection{Reduction to staircase}

\begin{lemma}
\label{lem:staircase-reduction}
Let \(P\) be a Dyck skeleton ending at \(m=\max(P)\), and let \(X\) be a multiset of nonnegative integers with \(|X|=2k+1\) and assume there is some $h$ such that 
$(P,X,h)$ supports an upper or lower.  
If \(h\ge \mathcal A(P,X)\) and \(2h>\mathcal B(P,X)\), then \(2h>\mathcal B(S_m,X)\) where  \(S_m=012\cdots m\).
\end{lemma}

\begin{proof}

Let \(Y\) refer to the entries that are not the rightmost instance of their value in $P$\footnote{We may assume $Y$ is not empty
since the result is immediate otherwise}. 
For each dinv pair in $P$ declare its owner to be the leftmost entry of that pair.
Since $P$ ends on its maximum every such owner belongs to $Y$.
Therefore  $\dinv(P)+\sum(Y)$ can be calculated by adding the value of each element in $Y$ to
the number of dinv pairs it owns.  Suppose an element of value $j$ in $Y$ has $i_0$ values in $Y$ that are less
than it (note it also has $j$ values not in $Y$ less than it) 
and has $i$ instances of $j$ to its right in $P$.  Thus in $P$ there are $i_0+j$ values less than it
that it could form unequal dinv pairs with, but 
Lemma~\ref{lem:full-skeleton-statistics}\footnote{The initial subsequence of $P$
ending on the rightmost element of $Y$ is a full Dyck skeleton. Suppose that subsequence ends on $t$; we show the next element is not $t+1$:
Otherwise, it is the last $t+1$ and $P[-1]=\max(P)$ implies all remaining are values greater than $t+1$ contradicting $t \in Y$.
Thus $t$ must be un-extractable in $P$ on account of at least two $t-1$s to its left so it is not (nor is any element to its left) extractable in the subsequence.}
 implies that $2j$ of them precede the element.  On the other hand the element owns exactly 
$i$ equal value dinv pairs.  Thus its total number of dinv pairs plus its value is at most $i_0+j-2j+i+j=i_0+i$.
Since the values of $i_0+i$ range from $1$ to $|Y|$ across the elements of $Y$ we get

\begin{equation}
\label{eq:staircase-prefix-bound}
\boxed{\dinv(P)+\sum Y
\le\sum_{i=1}^{|Y|}i=\binom{|Y|+1}{2}.}
\end{equation}

Next, let $t$ be the maximum value of $Y$ 
and set $a=|X_{\le t}|$, 
$b=|X_{>t}|$
and let \(K\) be the total number of pairs of entries within $Y$, between $Y$ and $X$, and between $Y$ and other elements of $P$:
\[ 
K=|Y|(m+a+b)+\binom{|Y|+1}{2}.
\]
 Define:
\[
\begin{aligned}
\Lambda&:=\mathcal A(S_m,X)-\mathcal A(P,X)
        =2\sum Y+\dinv(P)+C(Y,X)-K,\\
\Delta&:=\mathcal B(S_m,X)-\mathcal B(P,X)
        =\sum Y+2\dinv(P)+2C(Y,X)-K.
\end{aligned}
\]
By \eqref{eq:staircase-prefix-bound},
\begin{equation}
\label{eq:staircase-wall-changes}
\begin{aligned}
\Lambda&\le \sum Y-|Y|(m+a+b)+C(Y,X),\\
\Delta&\le\binom{|Y|+1}{2}-\sum Y-|Y|(m+a+b)+2C(Y,X).
\end{aligned}
\end{equation}

Since by   Lemma~\ref{lem:full-skeleton-statistics} all $0,1,\ldots t$ occur in $Y$ we have
\begin{equation}
\label{eq:staircase-y-sum-bounds}
\binom{t+1}{2}\le \sum Y\le t|Y|-\binom{t+1}{2}.
\end{equation}
Next we bound $C(Y,X)$. We need only consider the $a-x_t$ values of $X$ that are less than $t$ and the $x_t$ values equal to $t$. 
Since each of $0,1,\ldots,t$ occurs in $Y$, $t-1$ of these do not pair with the former and $t$ do not pair with the latter.  Thus
\begin{equation}
\label{eq:staircase-cross-count-bound}
C(Y,X)\le  (|Y|-(t-1))(a-x_t) + (|Y|-t)x_t = (|Y|-t+1)a-x_t.
\end{equation}

Substituting \eqref{eq:staircase-y-sum-bounds} and \eqref{eq:staircase-cross-count-bound} into \eqref{eq:staircase-wall-changes} gives
\[
\begin{aligned}
\Delta\le{}&
\binom{|Y|+1}{2}-\binom{t+1}{2}-|Y|(m+a+b)
+2(|Y|-t+1)a-2x_t.
\end{aligned}
\]
Factoring the binomials, writing \(|Y|=(|Y|-t)+t\) in the third term,
and splitting off \(2a\) from the fourth term gives
\[
\begin{aligned}
\Delta\le{}&\frac{(|Y|-t)(|Y|+t+1)}{2}-(|Y|-t)(m+a+b)-t(m+a+b)+2a(|Y|-t)+2a-2x_t.
\end{aligned}
\]
Isolating the coefficient of \(|Y|-t\) and rearranging $t$'s and $a$'s gives
\begin{equation}
\label{eq:staircase-delta}
\begin{aligned}
\Delta\le{}&(|Y|-t)
\left[\frac{|Y|-t+1}{2}-(m-t+b-a)\right]-t(m+a+b)+2(a-x_t).
\end{aligned}
\end{equation}

Next we derive a bound on \(\sum X\). The entries less than \(t\) have sum at most \((t-1)(a-x_t)\). The entries equal to \(t\) have sum at most \(tx_t\). Let \(b_1\) be the number of entries greater than \(t\) but at most \(m\); they sum to at most \(b_1m\). Let \(b_2\) be the number of entries greater than \(m\). Since \(\max(P)=m\) and \((P,X,h)\) supports a lower or upper for some \(h\) no value \(m<v<\max(X)\) can be absent from $X$ so their sum is at most \( (m+1)+(m+2)+\cdots+(m+b_2)=b_2m+\binom{b_2+1}{2}\). Since \(b_1+b_2=b\), these four bounds give
\begin{equation}
\label{eq:staircase-tail-sum}
\sum X\le(t-1)a+x_t+mb+\binom{b+1}{2}.
\end{equation}

Next use \(\mathcal A(P,X)=\mathcal A(S_m,X)-\Lambda\),
\[
\mathcal B(S_m,X)
=2\mathcal A(S_m,X) +3\left(\binom{m+1}{2}+\sum X+k\right)-\binom{m+a+b+1}{2},
\]
and $a+b-1=2k$ to get
\[
\begin{aligned}
&4\mathcal A(P,X)-2\mathcal B(S_m,X)\\
&\quad=2\binom{m+a+b+1}{2}-6\binom{m+1}{2}
-6\sum X-3(a+b-1)-4\Lambda\\
&\quad\ge 2\binom{m+a+b+1}{2}-6\binom{m+1}{2}-3(a+b-1)\\
&\qquad-6\Biggl[(t-1)a+x_t+mb
+\binom{b+1}{2}\Biggr]\\
&\qquad-4\Biggl[t|Y|-\binom{t+1}{2}+(|Y|-t+1)a
-x_t-|Y|(m+a+b)\Biggr].
\end{aligned}
\]
Here \eqref{eq:staircase-tail-sum} bounds \(\sum X\), while the first line of \eqref{eq:staircase-wall-changes}, followed by the
upper bounds for \(\sum Y\) and \(C(Y,X)\) in \eqref{eq:staircase-y-sum-bounds} and \eqref{eq:staircase-cross-count-bound}, respectively, bounds \(\Lambda\).

Next we prove
\(\Delta>0\Longrightarrow4\mathcal A(P,X)-2\mathcal B(S_m,X)>0\).
First we note that
\[
\begin{aligned}
&2\binom{m+a+b+1}{2}-6\binom{m+1}{2}-6mb
-6\binom{b+1}{2}+4\binom{t+1}{2}\\
&\quad=2\binom{a+1}{2}+2a(m+b)+2\binom{m+b+1}{2}
-6\binom{m+b+1}{2}+4\binom{t+1}{2}\\
&\quad=2\binom{a+1}{2}+2a(m+b)-2(m+b+t+1)(m+b-t).
\end{aligned}
\]
In the first equality, we split the first binomial and combine the next three
terms. In the second, we combine the last three terms and factor their sum. 
Substituting this equality into the preceding inequality and collecting
the \(|Y|\)-terms and the \(x_t\)-terms gives
\[
\begin{aligned}
&4\mathcal A(P,X)-2\mathcal B(S_m,X)\\
&\quad\ge 2\binom{a+1}{2}+2a(m+b)-2(m+b+t+1)(m+b-t)\\
&\qquad-3(a+b-1)-6(t-1)a-4(1-t)a\\
&\qquad+(-6+4)x_t+\bigl[-4t-4a+4(m+a+b)\bigr]|Y|\\
&\quad=2\binom{a+1}{2}+2a(m+b)+(m+b-t)\left[4|Y|-2(m+b+t+1)\right]\\
&\qquad-3(a+b-1)-2(t-1)a-2x_t\\
&\quad=(m+b-t)
\left[4(|Y|-t)-2-2(m+b-t-a)\right]+a^2-3b+3-2x_t.
\end{aligned}
\]
 Now
\[
\begin{aligned}
-3b+a^2+3-2x_t
&\geq-3(m-t+b)+(a-1)^2+2,
\end{aligned}
\]
since $m \geq t$ and $a \geq x_t$. Therefore
\begin{equation}
\label{eq:staircase-other-wall}
\begin{aligned}
&4\mathcal A(P,X)-2\mathcal B(S_m,X)
\ge(a-1)^2+2+
(m-t+b)
\left[4(|Y|-t)-5-
2(m-t+b-a)\right].
\end{aligned}
\end{equation}

\textbf{If \(a\ge m-t+b\),}
the prefactor in \eqref{eq:staircase-other-wall} lies between zero and \(a\),
and the bracket is at least \(-1\), since \(|Y|-t\ge1\).
Their product is therefore at least \(-a\), so
\[
4\mathcal A(P,X)-2\mathcal B(S_m,X)
\ge\bigl(a-\tfrac32\bigr)^2+\tfrac34>0.
\]

\textbf{If \(a<m-t+b\),}
\(\Delta>0\) will make the bracket in \eqref{eq:staircase-other-wall} positive. Indeed,
the term \(-t(m+a+b)+2(a-x_t)\) in \eqref{eq:staircase-delta}
is zero for \(t=0\), since \(a=x_0\),
and negative for \(t\ge1\), since \(m+b>a\).
Thus \(\Delta>0\) and \(|Y|-t\ge1\) force the bracket in \eqref{eq:staircase-delta} to be positive. Hence
\begin{equation}
\label{eq:staircase-integrality}
\begin{aligned}
|Y|-t+1&>2(m-t+b-a),\\
|Y|-t&\ge2(m-t+b-a).
\end{aligned}
\end{equation}
By \eqref{eq:staircase-integrality}, the bracket in \eqref{eq:staircase-other-wall} is at least
\(6(m-t+b-a)-5\ge1\).
The product is therefore nonnegative, making the whole expression positive.

Thus \(\Delta>0\) gives \(\mathcal B(S_m,X)<2\mathcal A(P,X)\),
while \(\Delta\le0\) gives \(\mathcal B(S_m,X)\le \mathcal B(P,X)\), whence the lemma follows.

\end{proof}

\subsection{Translation}

\begin{definition}
A \emph{simple lower} (respectively, \emph{simple upper}) is a pair \((A,B)\) of integer sequences such that \(0:A\) is affine, \(B:(-1)\) is reverse affine, and \((|A|,|B|)=(k+1,k)\) (respectively, \((k,k+1)\)) for some \(k\ge0\).
\end{definition}

\begin{definition}
A \emph{simple key} is a pair \((X_0,h)\), where \(X_0\) is a multiset of integers of size \(2k+1\) for some \(k\ge0\), and \(h\) is a nonnegative integer. A simple lower or upper \((A,B)\) is associated to \((X_0,h)\) if \(X_0\) is the multiset union of \(A\) and \(B\) and \(h=\dinv(A:B)\). A simple key is \emph{balanced} if it has equally many associated simple lowers and simple uppers.
\end{definition}

For \(|X_0|=2k+1\), define
\[
\beta(X_0)=\binom{2k+2}{2}+k-2\mu(X_0)-\sum X_0,
\]
where \(\mu(X_0)\) is the number of zeros in \(X_0\) plus twice the number of negative entries.

\begin{lemma}
\label{lem:simple-key-reduction}
If every simple key \((X_0,h)\) with \(2h>\beta(X_0)\) is balanced, then for every \((a,d)\) with \(a\in[d,M-2d)\), there are as many lowers of area \(a\) and deficit \(d\) as uppers of area \(a+1\) and deficit \(d\).
\end{lemma}

\begin{proof}
By the discussion at the beginning of Section~2, it suffices to balance every key \((P,X,h)\) satisfying \eqref{eq:weak-wall} and \eqref{eq:strict-wall}. A key supporting no lower or upper is automatically balanced. For any remaining key, put \(m=\max(P)\). Lemma~\ref{lem:reduction-bounds} leaves two possibilities.
Suppose first that \(P[-1]=m-1\) and \(\max(X)\le m\). Both endpoint conditions are then automatic and by Proposition~4.16 of \cite{Hawkes26} we may assume the second 
factor is affine rather than reverse affine. Thus it follows from Corollary~3.20 of \cite{Hawkes26} that the key is balanced. 
It remains to consider \(P[-1]=m\).  By Lemma~\ref{lem:staircase-reduction}, \(2h>\mathcal B(S_m,X)\) so it suffices to show balance on all keys $(S_m,X,h)$ with \(2h>\mathcal B(S_m,X)\).
Let \(X_0\) be the multiset obtained by subtracting \(m\) from every entry of \(X\). 
Since \(\dinv(S_m)=0\), \(\sum S_m=\binom{m+1}{2}\), \(\sum X=\sum X_0+m(2k+1)\), and \(C(S_m,X)=\mu(X_0)\), substitution gives
\[
\mathcal B(S_m,X)=\binom{2k+2}{2}+k-2\mu(X_0)-\sum X_0=\beta(X_0).
\]
Hence \(2h>\beta(X_0)\), so the simple key is balanced by hypothesis.  Since the lowers and uppers associated to $(S_m,X,h)$ are in bijection with the lowers and uppers associated to
$(X_0,h)$ this implies the former key is also balanced.
\end{proof}
\section{Balance}

By the reduction of the previous section, it remains to prove that every simple key $(X_0,h)$ with $\max(X_0)\ge1$ is balanced whenever
\[
2h>\beta(X_0).
\]
Indeed, if $\max(X_0)\le0$, then both endpoint conditions are automatic and, by Proposition~4.16 of \cite{Hawkes26}, we may assume the second factor is affine rather than reverse affine; balance then follows from Corollary~3.20 of \cite{Hawkes26}. We first handle the cases $k=0$ and $k=1$ directly. For $k\ge2$, we express the difference between the numbers of simple lowers and simple uppers as an alternating telescoping sum and reinterpret this sum as a signed count of $pTs$ objects. We then introduce a more general version of $pTs$ objects and construct an involution on them that gives the required cancellation. One inequality needed to complete this argument is proved separately in the appendix.

\subsection{The cases $k=0$ and $k=1$}

We consider balance on the finer data $(w,h)$, where $w$ is a word with content $X_0$ and $h=\dinv(w)$. We compare the numbers of simple lowers and simple uppers $(A,B)$ with $A:B=w$; each number is either zero or one. We call $w$ \emph{bad} if these numbers differ. Balance on every $(w,h)$ would imply balance on $(X_0,h)$ by summing over words with content $X_0$ and dinv $h$. We retain the assumption $\max(X_0)\ge1$.

For $k=0$, the only bad word is $(1)$, which represents a lower but not an upper. Here $h=0$ and $\beta(\{1\})=0$, so the required inequality $2h>\beta(X_0)$ fails.

Now let $k=1$ and write $w=(x,y,z)$. The word represents a lower precisely when
\[
x\le1,\qquad y\le x+1,\qquad z\le0,
\]
and an upper precisely when
\[
x\le1,\qquad y\le z+1,\qquad z\le0.
\]
Thus a bad word satisfies the common endpoint conditions and exactly one of the two conditions on $y$.

If $\max(X_0)\ge3$, neither representation is possible. If $\max(X_0)=2$, the only bad words are $(1,2,z)$ with $z\le0$. For $z<0$, we have $h=0$ and $\beta(X_0)=-z>0$. For $z=0$, we have $h=1$ and $\beta(X_0)=2$. In either case, $2h>\beta(X_0)$ fails.

It remains to consider $\max(X_0)=1$. Write $X_0=\{x,y,1\}$, where $x\le y\le1$. Suppose first that $y-x\ge2$. Then $x$ differs from both other entries by at least $2$, so no pair involving $x$ contributes to dinv, and hence $h\le1$. Moreover,
\[
\sum X_0+2\mu(X_0)=
\begin{cases}
x+6, & y=1,\\
x+7, & y=0,\\
x+y+9, & y\le-1.
\end{cases}
\]
Each expression is at most $5$, so
\[
2h\le2\le7-\sum X_0-2\mu(X_0)=\beta(X_0).
\]
Therefore no word in this case satisfies the required inequality.

We may now assume $y-x\le1$. If $y<0$, placing $1$ first gives both a lower and an upper, while placing it elsewhere gives neither. Thus no bad word occurs. If $y\ge0$, the remaining contents are
\[
\{-1,0,1\},\qquad
\{0,0,1\},\qquad
\{0,1,1\},\qquad
\{1,1,1\}.
\]
Inspection using the conditions on $x$, $y$, and $z$ shows that the only bad words are
\[
(0,1,-1)
\qquad\text{and}\qquad
(-1,1,0).
\]
The first represents only a lower and the second only an upper. Both have content $X_0=\{-1,0,1\}$ and $h=1$, and both satisfy $2h>\beta(X_0)$, since $\beta(X_0)=1$.

Thus, among the words under consideration satisfying the required inequality, balance on $(w,h)$ fails only at these two words. Their contributions are of opposite types and belong to the same simple key, so they cancel. Consequently, every simple key satisfying $2h>\beta(X_0)$ is balanced for $k=0$ and $k=1$.

\subsection{Telescoping}

Fix a simple key $(X_0,h)$ with $|X_0|=2k+1$ and $k\ge2$.
If it supports neither a simple lower nor a simple upper, it is already
balanced, so assume that at least one exists.

\begin{definition}
A sequence of integers $d_1,\ldots,d_n$ is \emph{dual} if
$d_{i+1}\ge d_i+2$, and \emph{reverse dual} if $d_i\ge d_{i+1}+2$.
\end{definition}

For $0\le i\le k+1$ and $0\le j\le k$, let $\mathrm{Low}_{ij}$ be
the set of pairs
\[
A=(x_1,\ldots,x_{k+1}),\qquad B=(y_{-k},\ldots,y_{-1})
\]
with content $X_0$ and $\dinv(A:B)=h$ such that, putting $x_0=0$
and $y_0=-1$,
\[
\begin{aligned}
x_0,\ldots,x_i&\text{ is dual},
&x_i,\ldots,x_{k+1}&\text{ is affine},\\
y_{-k},\ldots,y_{-j}&\text{ is reverse affine},
&y_{-j},\ldots,y_0&\text{ is reverse dual}.
\end{aligned}
\]
Define $\mathrm{Up}_{ij}$ in the same way, using
$A=(x_1,\ldots,x_k)$ and $B=(y_{-(k+1)},\ldots,y_{-1})$,
for $0\le i\le k$ and $0\le j\le k+1$. By convention, set
\[
\mathrm{Low}_{i,k+1}=\varnothing,
\qquad
\mathrm{Up}_{k+1,j}=\varnothing
\qquad(0\le i,j\le k+1).
\]
Thus $\mathrm{Low}_{00}$ and $\mathrm{Up}_{00}$ are precisely the
simple lowers and simple uppers associated to $(X_0,h)$.

For $0\le i,j\le k$, define
\[
\mathrm{LOW}_{ij}
=\mathrm{Low}_{ij}\cup\mathrm{Low}_{i,j+1}
 \cup\mathrm{Low}_{i+1,j}\cup\mathrm{Low}_{i+1,j+1},
\]
and define $\mathrm{UP}_{ij}$ similarly. These unions are disjoint.

The existence of a simple lower or upper on $(X_0,h)$ gives
$\max(X_0)\le k+1$, since $0:A$ is affine and $B:(-1)$ is reverse
affine. On the other hand, a pair in $\mathrm{Low}_{k+1,0}$ would
have $x_{k+1}\ge2k+2$, and a pair in $\mathrm{Up}_{0,k+1}$ would
have $y_{-(k+1)}\ge2k+1$. Thus $\mathrm{Low}_{k+1,0} =\mathrm{Low}_{0,k+1}=\mathrm{Low}_{k+1,k+1}= \emptyset$
and $\mathrm{Up}_{k+1,0}  = \mathrm{Up}_{0,k+1}  = \mathrm{Up}_{k+1,k+1}=\emptyset$ whence
\[
\begin{aligned}
\sum_{i,j=0}^{k}(-1)^{i+j}|\mathrm{LOW}_{ij}|
&=|\mathrm{Low}_{00}|,\\
\sum_{i,j=0}^{k}(-1)^{i+j}|\mathrm{UP}_{ij}|
&=|\mathrm{Up}_{00}|.
\end{aligned}
\]

Consequently, balance on $(X_0,h)$ is equivalent to
\begin{equation}
\label{eq:telescoping-balance}
\sum_{i,j=0}^{k}(-1)^{i+j}
\bigl(|\mathrm{LOW}_{ij}|-|\mathrm{UP}_{ij}|\bigr)=0.
\end{equation}
In the next subsection, we interpret this sum as a signed count
of $pTs$ objects.

\subsection{\texorpdfstring{$pTs$}{pTs} objects}

Retain the simple key $(X_0,h)$ and $k\ge2$ from the preceding
subsection. 

\begin{definition}
A \emph{Dyck tableau} is an integer filling of a Young diagram in
English convention whose rows are dual Dyck sequences, read from left to right, and whose
columns are affine Dyck sequences, read from bottom to top \cite[Definition~2.6]{Hawkes26}.
Its \emph{row-reading word} $\mathrm{RR}(T)$ reads these rows left to right from
bottom to top.
\end{definition}
\begin{definition}
A \emph{$pTs$ object} associated to $(X_0,h)$ is a triple $(p,T,s)$
where $0:p$ is dual,
$s:(-1)$ is reverse dual with $|p| \neq |s|$ and $T$ is a Dyck tableau with column lengths 
\[(k-\min(|p|,|s|), k+1-\max(|p|,|s|))
\]
and where the combined content of these objects is $X_0$ and 
$\dinv(p:\mathrm{RR}(T):s)=h$. 
\end{definition}

\begin{lemma}\label{alt}
For a fixed simple key $(X_0,h)$ and fixed $0\leq i,j\leq k$,
\[ |\mathrm{LOW}_{ij}|-|\mathrm{UP}_{ij}| = (-1)^{\mathbf{1}_{\{i<j\}}} \#\{pTs:(|p|,|s|)=(i,j)\}. \]

\end{lemma}

\begin{proof}
We fix $p$ of length $i$ and $s$ of length $j$ such that  $0:p$ is dual and $s:(-1)$ is reverse dual and show that the contribution to both sides is the same. 
The contribution of $(p, s)$ to $|\mathrm{LOW}_{ij}|-|\mathrm{UP}_{ij}|$ we take to mean the number $L$ of ways to 
partition $X_0$ as $(p:x,y:s)$ where $\dinv(p:x:y:s)=h$, $|p:x|=k+1$, $|y:s|=k$, $x$ is affine and $y$ is reverse affine
 minus the 
number, $U$, of ways to 
partition $X_0$ as $(p:x',y':s)$ where $\dinv(p:x':y':s)=h$, $|p:x'|=k$, $|y':s|=k+1$, $x'$ is affine and $y'$ is reverse affine.
Note that, by definition of $\mathrm{LOW}_{ij}$ (resp. $\mathrm{UP}_{ij}$), there is no condition on the boundaries between $p$ and $x$ (resp. $x'$) nor between $y$ (resp. $y'$) and $s$
 so these elements belong in $\mathrm{LOW}_{ij}$ and $\mathrm{UP}_{ij}$ respectively (and each element of one of these is associated to some $(p, s)$).

By Proposition~4.16 of \cite{Hawkes26}, we may in fact assume that the reverse affine factors are affine. 
Next note that if $i=j$ then the difference is $0$
by Corollary~3.20 of \cite{Hawkes26} and there are, by definition, no $pTs$ objects. Put
\[
Y=X_0\setminus p\setminus s,
\qquad
h'=h-\dinv(p)-\dinv(s)-C(p,Y)-C(Y,s)-C(p,s).
\]

Now if $i \neq j$ then by Corollaries~3.19 and~3.20 of \cite{Hawkes26}
 $L$ (resp. $U$) is equal to the number of pairs of a Dyck
tableau and a reverse semistandard tableau of the same shape where the first  has content $Y$ and $\dinv$ $h'$ and the latter
has weight $(k+1-i,k-j)$ (resp. $(k-i,k+1-j)$).  Thus for each Dyck tableau of content $Y$ and dinv $h'$ we must count the difference between
the number of ways to fill its shape in a reverse semistandard way with 
weight $(k+1-i,k-j)$ versus weight $(k-i,k+1-j)$.

Now suppose that $i<j$ (resp. $i>j$) and consider two-column partitions of $|Y|$.
Any two column shape with more than $k-i$ (resp. $k-j$) rows can be filled exactly one way with either weight.
Any two column shape with fewer than $k-i$ (resp. $k-j$) rows can be filled in no ways by either weight. 
A two column shape with $k-i$ (resp. $k-j$) rows can be filled in one way by the second (resp. first) and in no way by the first (resp. second). 
Thus the difference is equal to the number of Dyck tableaux of shape  $(k-i,k+1-j)$ (resp. $(k-j,k+1-i)$) with content $Y$ and dinv $h'$ with negative (resp. positive) sign. 
This is precisely the count of all $pTs$ objects with our fixed $p$ and $s$ counted with sign equal to $(-1)^{\mathbf{1}_{\{|p|<|s|\}}}$.

\end{proof}

\begin{corollary}
\label{lem:pts-signed-count}
For the fixed simple key $(X_0,h)$,
\begin{equation}
\label{eq:pts-count}
\sum_{i,j=0}^{k}(-1)^{i+j}
\bigl(|\mathrm{LOW}_{ij}|-|\mathrm{UP}_{ij}|\bigr)
=
\sum_{(p,T,s)}(-1)^{|p|+|s|}(-1)^{\mathbf{1}_{\{|p|<|s|\}}},
\end{equation}
where the sum on the right is over all $pTs$ objects associated to
$(X_0,h)$. If the $pTs$ objects on $(X_0,h)$ are balanced, then the key is balanced.
\end{corollary}

\begin{proof}
The first statement follows from Lemma \ref{alt}. By \eqref{eq:telescoping-balance} and the equation above
the simple key $(X_0,h)$ is balanced if its $pTs$ objects
are balanced.
\end{proof}

\subsection{An involution}

Retain the simple key $(X_0,h)$ and $k\ge2$. As in the preceding
subsections, assume that the key supports a simple lower or upper,
so $\max(X_0)\le k+1$.

\begin{definition}
A \emph{$uv$-$pTs$ object} (resp.  \emph{flipped $uv$-$pTs$ object}) associated to $(X_0,h)$ is a triple $(p,T,s)$
where 
$0,p_1,\ldots, p_u$ (resp. $(-1),p_1,\ldots, p_u$) is affine,
$p_u,\ldots,p_{|p|}$ is dual,
$s_{-|s|},\ldots,s_{-v}$ is reverse dual,
$s_{-v},\ldots,s_{-1},(-1)$ (resp. $s_{-v},\ldots,s_{-1},0$) is reverse affine,
where  $|p| \neq |s|$ and $T$ is a Dyck tableau with column lengths 
\[(k-\min(|p|,|s|), k+1-\max(|p|,|s|))
\]
and where the combined content of these objects is $X_0$ and 
$\dinv(p:\mathrm{RR}(T):s)=h$. 
\end{definition}
In particular, $00$-$pTs$ objects are precisely the
$pTs$ objects of the preceding subsection,
while flipped $uv$-$pTs$ objects just flip the endpoint requirements.

For a multiset $Y$ and an integer $h$, write
\[
 h_Y^*:=C(Y,Y)-|Y|-h
\]
for the complementary $\dinv$ parameter. (If a sequence $S$ on $Y$ has $\dinv$ $h$, then read backwards it has $\dinv$ $h_Y^*$.)

\begin{lemma}
\label{lem:tableau-reversal}
For any shape $\lambda$, content $Y$, and $\dinv$ $h$, the number
of Dyck tableaux of shape $\lambda$, content $Y$, and $\dinv$ $h$
equals the number of rotated Dyck tableaux with the same content
and $\dinv$. The latter have shape $\lambda$ rotated by $180^\circ$,
rows that are reverse dual when read from left to right, and
columns that are reverse affine when read from bottom to top.
Their row-reading words are still read from left to right,
bottom to top.
\end{lemma}

\begin{proof}
Fix a composition $\alpha$ of $|Y|$ and consider the dual Dyck
factorizations of $Y$ with factor lengths $\alpha$ and $\dinv$ $h$.
By symmetry \cite[Corollary~3.19]{Hawkes26}, there are equally many
such factorizations with factor lengths $\alpha^{\mathrm{rev}}$
and $\dinv$ $h$. Reversing each individual factor gives the same
number of reverse dual Dyck factorizations with factor lengths
$\alpha^{\mathrm{rev}}$ and $\dinv$ $h$. Reversing the entire
configuration then gives the same number of dual Dyck
factorizations with factor lengths $\alpha$ and $\dinv$ $h_Y^*$.

Thus the dual Dyck symmetric functions associated to $Y$ and
$\dinv$ values $h$ and $h_Y^*$ are equal, so their Schur
expansions agree. By \cite[Corollary~3.19]{Hawkes26}, for each
$\lambda$ there are equally many Dyck tableaux of shape $\lambda$
and content $Y$ with $\dinv$ $h$ as with $\dinv$ $h_Y^*$.
Rotating the latter by $180^\circ$ gives a bijection with the
rotated Dyck tableaux of the lemma.
\end{proof}

\begin{lemma}
\label{lem:flipped-pts-reduction}
Fix $(X_0,h)$. 
Let $\mathcal{P}_{uv}$ be the set of $uv$-$pTs$ objects
associated to $(X_0,h)$, and let $\widetilde{\mathcal{P}}_{vu}$
be the set of flipped $vu$-$pTs$ objects associated to
$(X_0,h_{X_0}^*)$. Then
\begin{equation}
\label{eq:flipped-pts-count}
 \sum_{(p,T,s)\in\widetilde{\mathcal{P}}_{vu}}
 (-1)^{|p|+|s|+\mathbf{1}_{\{|p|<|s|\}}}
 =
 -\sum_{(p,T,s)\in\mathcal{P}_{uv}}
 (-1)^{|p|+|s|+\mathbf{1}_{\{|p|<|s|\}}}.
\end{equation}
\end{lemma}

\begin{proof}

By Lemma~\ref{lem:tableau-reversal} we can replace $\widetilde{\mathcal{P}}_{vu}$ with 
$\widetilde{\mathcal{P}}_{vu}^{\circ}$ whose elements differ from those of the former only in that we
replace "Dyck tableau" with "rotated Dyck tableau" in the definition.  But there is a direct bijection from  
$\mathcal{P}_{uv}$ to $\widetilde{\mathcal{P}}_{vu}^{\circ}$ given by rotating the entire configuration by $180^{\circ}$.  
This exactly reverses the reading word so it sends $h$ to $h_{X_0}^*$.  The sum of the lengths of the two sequences
is maintained but the side which contains the longer of the two switches so the two summations are opposite.

\end{proof}

Our goal is to show that whenever the $22$-$pTs$ objects on $(X_0,h)$ are balanced, so are the
$pTs$ objects. It is possible to do this by showing that the $uv$-$pTs$ and $u'v'$-$pTs$
objects associated to $(X_0,h)$ are balanced for
\[
 (u,v,u',v')\in
 \{(0,0,0,1),(0,1,1,1),(1,1,1,2),(1,2,2,2)\}.
\]
By the lemma above, the lemma below is equivalent to doing this.

\begin{lemma}
\label{lem:adjacent-pts-balance}
For $(u,v)=(0,0),(1,1)$ the union of $uv$-$pTs$ and $u(v+1)$-$pTs$ objects is balanced over $X_0,h$.
For $(u,v)=(1,0),(2,1)$ the union of the flipped $uv$-$pTs$ and flipped $u(v+1)$-$pTs$ objects is balanced over $(X_0,h_{X_0}^*)$.

\end{lemma}

\begin{proof}
We will use dual Dyck insertion and reverse insertion as in the proof of
\cite[Theorem~3.18]{Hawkes26} to describe a sign changing involution on  
each union.  Leave $p$ and the last $v$ entries of $s$
fixed.  Reverse the remaining part of $s$ and insert it into $T$.
Mark boxes in the
enlarged tableau by the following rule. 

\emph{Case $|p|<|s|$.} 
\begin{enumerate}
\item
If $T$ was rectangular (or empty) and a box was added to
both columns $1$ and $2$, mark only the added boxes to the right of
column $2$.
\item Otherwise, mark the bottom box of column 2 if and only if no box was
added to column two, and mark all other added boxes outside column 2.
\end{enumerate}

\emph{Case $|p|>|s|$.} 
\begin{enumerate}
\item
If $T$ was rectangular (or empty) and no box was added to
either column $1$ or column $2$, mark both bottom boxes in columns 1 and 2 along with
all the added boxes.
\item Otherwise, mark the bottom box of column 1 if and only if
no box was added to column 1, and mark all other added boxes outside
of column $1$. 
\end{enumerate}

Reverse insert from the marked boxes. Reverse the recovered sequence
and place it before the fixed suffix to form the new $s$; the
remaining tableau is the new $T$.

The only possible failure in the $|p|<|s|$ case is when there is no box in the
second column after the insertion. Since $|p|<|s|$, there were $k+1-|s|$ (and also $0$) entries
in column 2 before insertion.
Thus $|s|=k+1 \geq 3$ and since $v \leq 1$, at least two entries were inserted
contradicting the fact column 2 was empty after insertion.  

The only possible failure in the $|p|>|s|$ case is when there are no boxes at all 
after insertion.  In this case $|s|=k \geq 2$ and since $v \leq 1$ there was at 
least one entry inserted contradicting there being no boxes. 

Next, the marked boxes form a horizontal strip, and their removal leaves
the required two-column shape. The recovered part of $s$ is reverse
dual and the fixed part remains reverse affine; no condition on their juncture is required.
Content and $\dinv$ are preserved by insertion and reverse insertion within the 
active parts and so also on the whole objects.

Except in the two rectangular cases, $|s|$ changes by $1$ and the
same factor remains longer. In the rectangular cases, $|s|$ changes
by $2$ and the longer factor changes. In either case the sign is
reversed. Reinserting the recovered sequence gives the same
enlarged tableau, and the marking rule then selects the original
insertion strip. Hence the map is a sign changing involution, proving balance
of the union.

\end{proof}

\begin{corollary}\label{22}
For a simple key $(X_0,h)$ if the $22$-$pTs$ objects are balanced then the key is balanced.
\end{corollary}
\begin{proof}
By the discussion before Lemma~\ref{lem:adjacent-pts-balance} $22$-$pTs$ balance implies $pTs$ balance.
By Corollary~\ref{lem:pts-signed-count} this implies the key is balanced.
\end{proof}

\subsection{An inequality}
For any multiset of integers, $X_0$, define
\[
\begin{aligned}
 h_z(X_0)
 &=\max\{t:(X_0,t)\text{ supports a simple lower or upper}\},\\
 h_0(X_0)
 &=\max\{t:(X_0,t)\text{ supports a $pTs$ object}\},\\
 h_2(X_0)
 &=\max\{t:(X_0,t)\text{ supports a $22$-$pTs$ object}\}.
\end{aligned}
\]
We take the maximum of an empty set to be $-\infty$.
\begin{lemma}\label{ineq}
For any multiset of integers $X_0$, 
\begin{equation}
\label{eq:nonexistence-bound}
 2\min\{h_z(X_0),h_0(X_0),h_2(X_0)\}
 \le\beta(X_0).
\end{equation}
\end{lemma}

\begin{proof}
See appendix.
\end{proof}

\begin{theorem}
\label{thm:flat-middle}
For any  $n$ and  $d$, the number of Dyck paths of semilength $n$, deficit $d$, and
area $a$ is the same for every integer
\[
 a\in[d,\binom n2-2d].
\]
\end{theorem}

\begin{proof}
By Lemma~\ref{lem:simple-key-reduction} it suffices to show there are an equal number of lowers and uppers
on each simple key 
$(X_0,h)$ satisfying
\[
 2h>\beta(X_0).
\]
This has already been established for $k \leq 1$. Otherwise, Lemma~\ref{ineq} implies that at least one of the following
holds: The key supports neither a simple lower nor a
 simple upper in which case the key is automatically balanced. The key supports no $pTs$ objects in which case the key is balanced by Corollary~\ref{lem:pts-signed-count}.
The key supports no $22$-$pTs$ objects in which case the key is balanced by Corollary~\ref{22}.

\end{proof}

\section{Bijection}
\label{sec:bijection}

We now turn the cancellations in the proof of
Theorem~\ref{thm:flat-middle} into explicit bijections between
lowers and uppers whose area differs by $1$ on keys satisfying \eqref{eq:weak-wall} and
\eqref{eq:strict-wall}. Composing these with $\phi$ and $\psi$ from
Section~1 gives bijections between Dyck paths of consecutive areas, and hence the
required involution inside the treated region.  Our construction is an explicit constructive algorithmic matching: 
at no point do we use equality of cardinalities to define a bijection by arbitrarily ordering two equinumerous sets and matching their $i^{th}$ elements.

\paragraph{Finite cancellation.}

Suppose that we have bijections of finite sets
\[
 f:A\sqcup C\longrightarrow B\sqcup D,
 \qquad g:C\longrightarrow D.
\]
Starting at $x\in A$, apply $f$. Whenever the result belongs to $D$,
apply $g^{-1}$ followed by $f$, and continue until the result belongs
to $B$. Denote this map by $\operatorname{Cancel}(f,g)$. It is a
bijection from $A$ to $B$, with inverse
$\operatorname{Cancel}(f^{-1},g^{-1})$
\cite{GarsiaMilne81,Doyle19}.

\subsection{Affine Dyck insertion}

\begin{lemma}[Affine Dyck insertion]
\label{lem:affine-insertion}
Affine Dyck factorizations of fixed integer content $X_0$ and
$\dinv$ $h$ are in weight-preserving bijection with pairs $(T,Q)$
of the same shape, where $T$ is a Dyck tableau with this content
and $\dinv$, and $Q$ is reverse semistandard: rows increase strictly
and columns increase weakly. 
\end{lemma}

\begin{proof}
Put $m=|X_0|$ and $J=\{1,\ldots,m-1\}$. For a word
$w=(w_1,\ldots,w_m)$ of content $X_0$ and $\dinv$ $h$, define
\[
 \operatorname{Aff}(w)=\{i\in J:w_{i+1}\le w_i+1\},
 \qquad \operatorname{Dual}(w)=J\setminus\operatorname{Aff}(w).
\]
Consider pairs $(T,V)$ where $T$ is a Dyck tableau of content $X_0$
and $\dinv$ $h$, and $V$ is a standard tableau of the same shape.
Write $\operatorname{Des}(V)$ for the positions $i$ where $i+1$
lies in a lower row, and
$\operatorname{Asc}(V)=J\setminus\operatorname{Des}(V)$.

For  $Z\subseteq J$, cutting $w$ into factors at $Z$, applying dual Dyck
insertion \cite[Theorem~3.18]{Hawkes26}, and standardizing equal
recording labels left to right gives a bijection
\[
 I_Z:\{w:\operatorname{Aff}(w)\subseteq Z\}
 \longleftrightarrow
 \{(T,V):\operatorname{Des}(V)\subseteq Z\}.
\]
We construct refined bijections 
\[
 E_Z:\{w:\operatorname{Aff}(w)= Z\}
 \longleftrightarrow
 \{(T,V):\operatorname{Des}(V)= Z\}
\]
by induction on $|Z|$.
For $Z=\varnothing$, the bijection is obvious.
Suppose $E_{Z'}$ is constructed for $|Z'|<|Z|$.
Now set
\[
\begin{aligned}
 A&=\{w:\operatorname{Aff}(w)=Z\},
 &B&=\{(T,V):\operatorname{Des}(V)=Z\},\\
 C&=\{w:\operatorname{Aff}(w)\subsetneq Z\},
 &D&=\{(T,V):\operatorname{Des}(V)\subsetneq Z\},\\
 f&=I_Z,
 &g&=\bigsqcup_{Z'\subsetneq Z}E_{Z'}.
\end{aligned}
\]
By induction and finite cancellation this constructs $E_Z$ for all $Z$.
We may think of $E_{Z^c}$ as matching $\{w: \operatorname{Dual}(w)=Z\}$ to
$\{(T,V): \operatorname{Asc}(V)=Z\}$.
For $Z\subseteq J$, taking the disjoint union gives
\[
 \bigsqcup_{Z'\subseteq Z}E_{(Z')^c}:
 \{w:\operatorname{Dual}(w)\subseteq Z\}
 \longleftrightarrow
 \{(T,V):\operatorname{Asc}(V)\subseteq Z\}.
\]
For $Z=\{z_1<\cdots<z_\ell\}\subseteq J$, this is equivalent to a
bijection from affine factorizations of weight
\[
\nu= (z_1,z_2-z_1,z_3-z_2,\ldots,m-z_\ell)
\]
to pairs $(T,Q)$ of the same shape, where $T$ is a Dyck tableau and
$Q$ is a reverse semistandard tableau (formed by filling boxes with the ordered multiset corresponding to $\nu$ in the order specified by $V$)
of the same shape, and weight $\nu$. 
Since $Z$ is arbitrary, the result follows.
\end{proof}

Write $\mathrm{fw}$ for the reverse affine to affine bijection of
\cite[Proposition~4.16]{Hawkes26}, and $\mathrm{bk}=\mathrm{fw}^{-1}$.

\begin{lemma}
\label{lem:mixed-cut}
Fix $X_0$, $h$, and $a,b\ge0$ with
$a+b=|X_0|$. There is a bijection $\sigma_{a,b}$:
\[
\left\{(A,B):
\begin{array}{c}
A\text{ affine},\ B\text{ reverse affine},\\ (|A|,|B|)=(a,b),\\
\operatorname{cont},\dinv(A:B)=X_0, h
\end{array}
\right\}
\longleftrightarrow
\left\{T:
\begin{array}{c}
\operatorname{cont}(T)=X_0,\  \dinv(T)=h,\\
\#\operatorname{cols}(T)\le2,\ 
\#\operatorname{rows}(T)\ge\max(a,b)
\end{array}
\right\}.
\]
In particular, if $|X_0|=2k+1$, there is a bijection $\omega$:
\[
\left\{(A,B):
\begin{array}{c}
A\text{ affine},\ B\text{ reverse affine},\\ 
(|A|,|B|)=(k+1,k),\\
\operatorname{cont}, \dinv(A:B)=X_0, h
\end{array}
\right\}
\longleftrightarrow
\left\{(A',B'):
\begin{array}{c}
A'\text{ affine},\ B'\text{ reverse affine},\\ 
(|A'|,|B'|)=(k,k+1),\\
\operatorname{cont},\ \dinv(A':B')=X_0,h
\end{array}
\right\}.
\]
\end{lemma}

\begin{proof}
Compose $\mathrm{fw}$ on the second factor with
Lemma~\ref{lem:affine-insertion}, and use $\mathrm{bk}$ in the inverse.
A reverse semistandard recording tableau of weight $(a,b)$ exists
exactly when the shape has at most two columns and at least
$\max(a,b)$ rows, and in that case it is unique. This gives
$\sigma_{a,b}$. The final statement follows by composing
$\sigma_{k+1,k}$ with $\sigma_{k,k+1}^{-1}$.
\end{proof}

\subsection{Complementation}

\begin{lemma}[Complement map]
\label{lem:dyck-rotation}
Fix a multiset $Y$ and $\dinv$ $h$. There is a content and shape preserving bijection $\chi_Y$ on Dyck tableaux with at most two columns that sends 
$h$ to $h_Y^*$.
\end{lemma}

\begin{proof}

Put $m=|Y|$, and let $\mathcal T_a(Y,h)$ denote the Dyck tableaux
of content $Y$ and $\dinv$ $h$ having exactly $a$ rows and at most two columns. Write
\[
h_Y^*=C(Y,Y)-m-h.
\]

For $A$ affine and $B$ reverse affine, define
\[
R(A,B)=(B^{\mathrm{rev}},A^{\mathrm{rev}}).
\]
Then $R$ preserves content, interchanges the factor lengths, and
sends $\dinv$ $h$ to $h_Y^*$.

For $a\ge m-a$, Lemma~\ref{lem:mixed-cut} identifies affine/reverse affine pairs of
lengths $(a,m-a)$ with
\[
\bigsqcup_{r\ge a}\mathcal T_r(Y,h),
\]
and affine/reverse affine pairs of lengths $(m-a,a)$ with  \[
\bigsqcup_{r\ge a}\mathcal T_r(Y,h_Y^*).
\]Hence
\[
f_a=\sigma_{m-a,a}\circ R\circ\sigma_{a,m-a}^{-1}
\]
is a bijection
\[
\bigsqcup_{r\ge a}\mathcal T_r(Y,h)
\longleftrightarrow
\bigsqcup_{r\ge a}\mathcal T_r(Y,h_Y^*).
\]

Suppose by induction that for every $r>a$ we have already
constructed a shape-preserving bijection
\[
\chi_Y:\mathcal T_r(Y,h)
\longleftrightarrow
\mathcal T_r(Y,h_Y^*).
\]
Thus, applying finite cancellation to $f_a$ and
\[
\bigsqcup_{r>a}\chi_Y:
\bigsqcup_{r>a}\mathcal T_r(Y,h)
\longleftrightarrow
\bigsqcup_{r>a}\mathcal T_r(Y,h_Y^*)
\]
gives
\[
\chi_Y:\mathcal T_a(Y,h)
\longleftrightarrow
\mathcal T_a(Y,h_Y^*).
\]

\end{proof}

\subsection{Realizing the telescoping cancellation}

Fix $(X_0,h)$ as in Section~3.2. Set
\[
 \mathcal Q=\bigsqcup_{i,j=0}^k
 \bigl(\mathrm{LOW}_{ij}\sqcup\mathrm{UP}_{ij}\bigr),
\]
where each $q \in \mathcal Q$ is considered distinct due to its association with a particular $\mathrm{LOW}_{ij}$ or $\mathrm{UP}_{ij}$. 
 Set 
$\varepsilon(q)=(-1)^{i+j}$  for  each $q \in \mathrm{LOW}_{ij}$ and
$\varepsilon(q)=-(-1)^{i+j}$  for each  $q \in \mathrm{UP}_{ij}$.

\begin{lemma}
\label{lem:outer-cut-matching}
There is a key-preserving sign-reversing matching $\eta$ on
$\mathcal Q$ leaving precisely the simple lowers and uppers
unmatched.
\end{lemma}

\begin{proof}
 Consider the elements of $\mathcal Q$ that are elements of $\mathrm{Low}_{ab}$. If $ 1 \leq a \leq k$, match the elements of $\mathrm{Low}_{ab} \in \mathrm{LOW}_{(a-1)c}$ to the elements of $\mathrm{Low}_{ab} \in \mathrm{LOW}_{ac}$
where $c=b$ or also $c=b-1$ if $b>0$.  If $a=0$ and $1 \leq b \leq k$,
match the elements of $\mathrm{Low}_{ab} \in \mathrm{LOW}_{a(b-1)}$ to the elements of $\mathrm{Low}_{ab} \in \mathrm{LOW}_{ab}$.
This matches all elements of $\mathcal Q$ except elements of $\mathrm{Low}_{ab}$ for $(a,b)$ both $0$ or at least one equal to $k+1$, but $\mathrm{Low}_{ab}$
is empty when at least one of $a,b$ equals $k+1$ for $k \geq 2$ (Section~3.2). Thus the only unmatched elements are those of $\mathrm{Low}_{00}$.  
Extend the matching to uppers in the same way.

\end{proof}

For $pTs$ objects, define
$\varepsilon(p,T,s)=
(-1)^{|p|+|s|+\mathbf1_{\{|p|<|s|\}}}$.

\begin{lemma}
\label{lem:middle-cut-matching}
Let $\mathcal E\subseteq\mathcal Q$ consist of the occurrences
in $\mathrm{LOW}_{ij}\sqcup\mathrm{UP}_{ij}$ whose middle pairs
insert to tableaux with first column of the length
$c = \max(k-i,k-j)$.
There is a complete key-preserving sign-reversing matching
$\kappa$ on $\mathcal Q\setminus\mathcal E$, pairing lowers
with uppers of the same indices.
The excluded occurrences admit a key-preserving sign-preserving
bijection $\pi$ from $\mathcal E$ to the set of $pTs$ objects for the fixed key.
\end{lemma}

\begin{proof}

Let $\kappa$ be the map from $\mathrm{LOW}_{ij} \setminus \mathcal E$ to 
$\mathrm{UP}_{ij} \setminus \mathcal E $ defined by acting with $\sigma_{k-i,k+1-j}^{-1} \circ \sigma_{k+1-i,k-j}$ on the inner factors and fixing the outer factors.  The exclusion of pairs mapping under
$\sigma$ to tableaux with first column length $c$ ensures this is a bijection. It clearly reverses sign and preserves content and $\dinv$. 

Let $\pi$ be the map that maps elements of $\mathcal E$ to $pTs$ objects by acting with 
$\sigma_{k-i,k+1-j}$ if $i<j$ and with $\sigma_{k+1-i,k-j}$ if $i>j$ on the inner factors and preserving the outer factors.  (Note that no elements with $i=j$ exist in $\mathcal E$.)

\end{proof}

\subsection{Constructive tableau cancellation}

Retain the simple key and $k\ge2$. Write $\mathcal P_{uv}$ for
its $uv$-$pTs$ objects, using the same formula for $\varepsilon$
on every family.

\begin{lemma}
\label{lem:tableau-matchings}
For $j=0,1$, there are key-preserving, $\varepsilon$-reversing
involutions
\[
 \rho_j\text{ on }\mathcal P_{jj}\sqcup\mathcal P_{j,j+1},
 \qquad
 \lambda_j\text{ on }\mathcal P_{j,j+1}\sqcup\mathcal P_{j+1,j+1}.
\]
\end{lemma}

\begin{proof}
Take $\rho_j$ from Lemma~\ref{lem:adjacent-pts-balance}.
For a triple with middle content $Y$, put
\[
 F(p,T,s)=
 \bigl(s^{\mathrm{rev}},
      \chi_Y(T),
       p^{\mathrm{rev}}\bigr).
\]
By Lemma~\ref{lem:dyck-rotation}, this is a bijection from
$\mathcal P_{uv}$ to the flipped $vu$ family at
$h_{X_0}^*$. It reverses $\varepsilon$, since the outer lengths
are exchanged; its inverse uses
$\chi_Y^{-1}$ on the tableau.
Let $\widetilde\rho_j$ be the involution of
Lemma~\ref{lem:adjacent-pts-balance} on the flipped
$(j+1)j$ and $(j+1)(j+1)$ families, and set
$\lambda_j=F^{-1}\circ\widetilde\rho_j\circ F$.
Content is preserved,
$\dinv$ is complemented twice, and $\varepsilon$ is reversed
three times.
\end{proof}

\subsection{The lower--upper bijection}

For a key $(P,X,h)$, write $\mathcal L(P,X,h)$ and
$\mathcal U(P,X,h)$ for the sets of associated lowers and uppers.

\begin{theorem}
\label{thm:lower-upper-bijection}
For every key $(P,X,h)$ satisfying \eqref{eq:weak-wall} and
\eqref{eq:strict-wall}, there are explicit inverse bijections
\[
\operatorname{up}:\mathcal L(P,X,h)\longrightarrow\mathcal U(P,X,h),
\qquad
\operatorname{down}:\mathcal U(P,X,h)\longrightarrow\mathcal L(P,X,h).
\]
\end{theorem}

\begin{proof}
Fix a key $(P,X,h)$ satisfying \eqref{eq:weak-wall} and
\eqref{eq:strict-wall}. Set $m = \max(P)$.

If $m>P[-1]$, then Lemma~\ref{lem:reduction-bounds} gives
\[
m=P[-1]+1,\qquad \max(X)\le P[-1]+1.
\]
Both endpoint conditions are therefore automatic. Apply
$\omega$ from Lemma~\ref{lem:mixed-cut} to $(A,B)$, retaining
$P$, to define $\operatorname{up}$; applying $\omega^{-1}$
defines $\operatorname{down}$.

Now assume $m=P[-1]$, and let $X_0=X-m$.
By Lemma~\ref{lem:staircase-reduction} and the translation in
Section~2.3, we have $2h>\mathcal B(S_m,X)=\beta(X_0)$.
Any bijection between the simple lowers and uppers on $(X_0,h)$
lifts to $(P,X,h)$ by subtracting $m$ from both factors, applying
the bijection, and adding $m$ back, with $P$ unchanged. It remains to construct these
simple-key bijections. Write $|X_0|=2k+1$.

If $\max(X_0)\le0$, both endpoint conditions in the definitions
of simple lowers and simple uppers are automatic. Use the
bijection $\omega$ of Lemma~\ref{lem:mixed-cut}, with inverse
$\omega^{-1}$.

Now suppose $\max(X_0)\ge1$. By Section~3.1, the case $k=0$ is vacuous,
and for $k=1$, that section gives an implicit bijection: outside of the simple lower $\bigl((0,1),(-1)\bigr)$ and simple 
upper $\bigl((-1),(1,0)\bigr)$ (which are matched) the map retains the underlying word and moves the middle 
entry between the left and right factors.

It remains to treat a simple key supporting at least one simple
lower or upper, with $k\ge2$ and $2h>\beta(X_0)$.
Use $\mathcal Q$, $\mathcal E$, $\eta$, $\kappa$, and $\pi$
from Lemmas~\ref{lem:outer-cut-matching} and
\ref{lem:middle-cut-matching}, and use the families
$\mathcal P_{uv}$ and involutions $\rho_j,\lambda_j$ from
Lemma~\ref{lem:tableau-matchings}. In particular,
$\mathcal P_{00}$ is the set of $pTs$ objects for this key.

Form the finite set
\[
\mathcal V=
\mathcal Q\sqcup\mathcal P_{00}\sqcup\mathcal P_{01}
\sqcup\mathcal P_{11}\sqcup\mathcal P_{12}\sqcup\mathcal P_{22}.
\]
 Give $\mathcal V$ the sign
\[
\widehat\varepsilon(z)=
\begin{cases}
\varepsilon(z),&z\in\mathcal Q,\\
-\varepsilon(z),&z\in\mathcal V\setminus\mathcal Q.
\end{cases}
\]

The bijection
$\pi:\mathcal E\longrightarrow\mathcal P_{00}$
becomes a matching $\widehat\pi$ on
$\mathcal E\sqcup\mathcal P_{00}$ by pairing each
$e\in\mathcal E$ with $\pi(e)$. Since $\pi$ preserves
$\varepsilon$, this matching reverses $\widehat\varepsilon$.
Define two matchings on $\mathcal V$ by
\[
\begin{aligned}
M_1&=\kappa\sqcup\widehat\pi\sqcup\lambda_0\sqcup\lambda_1,\\
M_2&=\eta\sqcup\rho_0\sqcup\rho_1.
\end{aligned}
\]
Both matchings reverse $\widehat\varepsilon$.

If $\mathcal P_{00}=\varnothing$,
then $\mathcal E=\varnothing$ by the bijection $\pi$. Thus
$\kappa$ is complete on $\mathcal Q$, while $\eta$ leaves
precisely the simple lowers and uppers unmatched. Finite
cancellation with $\kappa$ and $\eta$ gives a sign changing involution on $\mathrm{Low}_{00} \sqcup \mathrm{Up}_{00}$ which must be the desired
bijection.

Now suppose $\mathcal P_{00}\ne\varnothing$. This, the assumption that $(X_0,h)$ supports at least one simple lower or
upper, and Lemma~\ref{ineq} imply
$\mathcal P_{22}=\varnothing$. In this case the matching $M_1$ is complete
on $\mathcal V$, and $M_2$ leaves precisely the simple lowers
and uppers unmatched. Finite cancellation with $M_1$ and $M_2$
therefore gives a sign changing involution on $\mathrm{Low}_{00} \sqcup \mathrm{Up}_{00}$ which must be the desired
bijection.

\end{proof}

\subsection{An area and dinv interchanging involution}

We now give the main theorem. Set
\[
 \theta=\psi^{-1}\circ\operatorname{up}\circ\phi,
 \qquad
 \theta^{-1}=\phi^{-1}\circ\operatorname{down}\circ\psi.
\]
Set $\mathcal I(D)=D$
if $n\le1$ and otherwise
 $\mathcal I(D)=\theta^{\dinv(D)-\area(D)}(D)$,
where $\theta^0$ is the identity and negative powers denote
iterates of $\theta^{-1}$.

\begin{theorem}
\label{thm:area-dinv-involution}
The map $\mathcal I$ is an involution on Dyck paths satisfying
$\defc\le\min(\area,\dinv)$ that interchanges area and dinv.
\end{theorem}

\begin{proof}
Fix a semilength $n$ and a deficit, $d$.  It suffices to show that $\mathcal I$ is an involution on the set of 
Dyck sequences with $\defc$ $d$ and $\area \in [d, \binom n2-2d]$ that sends $\area$ to $\binom n2 -d - \area$.

Theorem~\ref{thm:lower-upper-bijection} and Lemma~\ref{lem:full-skeleton-statistics} show that $\theta$ is a bijection from $\defc$ $d$
sequences that have $\area a$ to $\defc$ $d$ sequences that have $\area$ $a+1$ for $a \in [d, \binom n2-2d)$.

Therefore applying $\theta$  $(\binom n2 -d -\area)-\area$ times (with negative values corresponding to applications of $\theta^{-1}$) sends $\area$ to $\binom n2 - d - \area$ and is clearly an involution. 
\end{proof}

\section{The full involution for defc $\leq 2n-8$}
\label{sec:extension}

Combining the results of \cite{Hawkes26},  Theorem~\ref{thm:area-dinv-involution}, and the partition bijection in \cite{LoehrWarrington09} with a few additional constructions will provide an $\area$ and $\dinv$ interchanging bijection on all Dyck paths with $\defc \leq 2n-8$.

Specifically, since deficit zero is completely covered by Theorem~\ref{thm:area-dinv-involution}, we may fix $1\le d\le2n-8$. The \emph{low-side} (paths with area at most $d$) arguments primarily rely on \cite{Hawkes26} while the \emph{high-side} (paths with $\dinv$ at most $d$) arguments make use of \cite{LoehrWarrington09}.  The middle region is covered by Theorem~\ref{thm:area-dinv-involution}.

\subsection{Full Dyck skeletons and binary partitions}
The string decomposition of \cite[Proposition~5.34]{Hawkes26} applies throughout the low side. Its successive steps preserve deficit and increase area by one. Stopping each string at area $d$ and omitting the exceptional step
\[
 0^{n-1}1\longmapsto0010^{n-4}1
\]
gives a decomposition of low-side Dyck paths where each string is indexed by a  full Dyck skeleton $F$ of deficit $d$ and covers each $\area$ in $[\area(F),d]$
except that the string starting at $0^{n-1}1$ contains only that path (of $\area$ equal to $1$).

Represent a nonempty partition by a binary word beginning with $0$ and ending with $1$ where the parts of the partition correspond to the number of $0$s preceding each $1$.
 The empty partition is represented by $0$. For a word $U$, write $\#_0(U)$ and $\#_1(U)$ for its numbers of zeros and ones. For a binary partition $\mu$, the number of parts is therefore $\#_1(\mu)$, while its size $|\mu|$ counts zero-before-one pairs.

We will now distinguish a subset of partitions of $d$ that label full Dyck skeletons. We ``cut'' the binary representation of a partition after $n-1$ symbols. 
Let $t$ be the number of ones after this cut and $r$ the number of $0$s after it.

 Call $\mu$ \emph{good} if $t=0$, or if
\[
 r \leq \#_1(\mu)-2t-2
 \qquad\hbox{and}\qquad
 \mu\hbox{ begins with }01^t.
\]
Call the remaining partitions \emph{bad}.

\paragraph{The map $\operatorname{wrap}$.}
Given a good partition $\mu$ written in binary, remove  $t$ ones immediately following the first $0$ and remove the suffix of $t+r$ entries.  Next, place as many $2$s immediately to the right of the second $1$ as the number of $1$s in the suffix to the left of all $0$s in the suffix. Then, for each $i \in [1,r]$, find the $(i+2)$th $1$ in the remaining word and add as
many $2$s immediately to its right as the $i^{th}$ $0$ in the removed suffix has $1$s immediately (i.e., before the next $0$) to its right. Finally prepend one $0$.

Define $\operatorname{wrap}^{-1}$ as follows.  Remove the initial $0$. Create a suffix of $0$s and $1$s, as follows: Initialize the suffix to have as many $1$s as the current sequence has $2$s. 
Before the first $1$, place as many $0$s as there are $1$s after the second $1$ and before the first $2$.  For each $i>1$, place before the $i^{th}$ $1$ as many $0$s as there are $1$s between the   $(i-1)$th and $i^{th}$ $2$ in the sequence.  Then remove all $2$s from the sequence, insert as many $1$s after the first $0$ as $2$s you removed and append the suffix.

\begin{lemma}
\label{lem:wrap}
The map $\operatorname{wrap}$ is a bijection between good partitions
and full Dyck skeletons of length $n$ with maximum entry at most $2$,
sending partition length to area and size to deficit.
\end{lemma}

\begin{proof}
First we establish that $\operatorname{wrap}$ is well defined.
Goodness guarantees that there are at least $t$ ones immediately
following the first zero, so the first removal is possible.
Goodness also ensures that, after
both removals, there are at least $r$ ones to the right of the
second remaining one. Thus all required insertion positions
are available. Placing each run of twos immediately after a one
ensures that the result is a Dyck sequence. Since we prepend an
additional zero, the result begins with $00$, and all twos are
inserted after the second one. Hence it is a full Dyck skeleton.

Next we establish that $\operatorname{wrap}^{-1}$ is well defined.
Clearly the construction produces a binary partition word $\mu$.
Removing the initial zero and replacing each removed two by
an initial one leaves $n-1$ entries before the reconstructed
suffix. Thus after applying $\operatorname{wrap}^{-1}$, the portion 
after the cut is precisely the suffix that we have constructed.
Thus the number of removed twos, $t$, is equal to the number of $1$s
 after the cut, and $\mu$ begins with $01^t$. Each zero after the cut
 corresponds to a one before the cut other
than the initial $t$ ones and the next two ones. There are
$\#_1(\mu)-t$ ones before the cut, so
\[
 r\le (\#_1(\mu)-t)-t-2.
\]
Thus $\mu$ is good. 

Next, it is apparent that the two maps are mutual inverses.

It remains to check the statistics. Partition length is the
number of ones, or equivalently the sum of the entries of the
binary sequence. The map $\operatorname{wrap}$ deletes $2t$
ones and adds $t$ twos, so this sum, which is the area of the
resulting Dyck sequence, does not change.

For size, we match the zero-before-one pairs of $\mu$ with the
deficit pairs of $F=\operatorname{wrap}(\mu)$. Each one which is not one of the 
first $t$ or last $t$ ones keeps the same number of
$0$--$1$ deficit pairs as it had zero-before-one pairs in
$\mu$: its original preceding zeros remain, and the prepended
zero makes each of them nonfirst. The $t$ pairs between the
first zero and the first $t$ ones of $\mu$ account for the
$1$--$2$ deficit pairs between the second one of $F$ and all
its twos. The zero-before-one pairs within the suffix become
the remaining $1$--$2$ deficit pairs, those not involving
the second one. Finally, the pairs between zeros before the
cut and ones after it become the $0$--$2$ deficit pairs.
These account for all deficit pairs of $F$, proving
$|\mu|=\defc(F)$.
\end{proof}

Next, recall the deficit pairs of \cite[Proposition~5.5]{Hawkes26}: pairs $i<j$ with $D_i>D_j+1$, or with $D_i<D_j$ and $D_i$ not the first occurrence of its value. Counting these pairs gives
the deficit of a Dyck sequence. 
Now, suppose that a full skeleton contains a $3$. By Lemma~\ref{lem:full-skeleton-statistics} there is a subword $0011223$ of such a skeleton,
 composed of the first two $0$s, the first two $1$s, the first two $2$s, and the first $3$ of the skeleton.  This subword gives 9 deficit pairs within itself. Each other $0$ pairs with the first $2$ and $3$ of the subword, each other $1$ with the second $0$ and the first $3$ of the subword, and every other entry with the second $0$ and second $1$. These pairs are distinct, giving $9+2(n-7)=2n-5$. Hence the maximum entry of a full Dyck skeleton with deficit at most $2n-8$ is certainly at most 2.

\begin{corollary}
\label{cor:wrap}
For $1\le d\le 2n-8$, $\operatorname{wrap}$ is a bijection from the good
partitions $\mu\vdash d$ to the full Dyck skeletons $F$ of length $n$ and
deficit $d$, with
\[
 \ell(\mu)=\area(F).
\]
\end{corollary}

\begin{proof}
By Lemma~\ref{lem:wrap}, $\operatorname{wrap}$ bijects good partitions with
full Dyck skeletons of maximum entry at most $2$, sending size to deficit and
length to area. The preceding argument shows that every full Dyck skeleton of
deficit at most $2n-8$ has maximum entry at most $2$.
\end{proof}

Corollary~\ref{cor:wrap} allows us to index the low-side strings by good
partitions of $d$. 
In particular every good partition $\mu$ except $\epsilon=0^{n-2}1$ (since  $\operatorname{wrap}(\epsilon)=0^{n-1}1$, the exceptional
full Dyck skeleton)
labels a string whose areas range from  
$\ell(\mu)=\#_1(\mu)$ to $d$, and $\epsilon$ labels a string starting and ending at $\area$ 1. Thus the low-side 
Dyck paths are exactly indexed by pairs 
$(\mu,s)$ with $\mu$ good and $\ell(\mu) \leq s \leq d$ for $\mu \neq \epsilon$ and $s=1$ otherwise.

\subsection{The strip map}
We now allow $\mu$ to be any partition of $d$, retaining the bounds
$\#_1(\mu)\le s\le d$. The pairs which do not label low-side paths
are those with $\mu$ bad, together with the pairs $(\epsilon,s)$
with $2\le s\le d$ when $d=n-2$. Call these pairs \emph{excluded}.
Put
\[
\begin{aligned}
 \mathscr R&=\{(\rho,s):|\rho|\le d-(n-1),\
                 \#_1(\rho)+1\le s\le d\},\\
 \mathscr C&=\{(\rho,s):|\rho|\le d-(n-2),\
                 \#_1(\rho)+2\le s\le d\},
\end{aligned}
\]
and consider
$\mathscr R\sqcup\mathscr C$ where otherwise indistinguishable elements are
identified by their membership in one of these sets.

Next we show that a bad word has exactly one $1$ after the cut. 

A partition with $t=0$ is good, so it suffices to rule out $t\ge2$.
First suppose that $\mu$ begins with $01^{t'}0$, where $t'<t$.
The first two suffix ones each pair with every prefix zero.
Every prefix one pairs with the first two zeros, except those
among the initial $t'$ ones, which pair with only the first zero.
Since the prefix has $n-1$ entries, these contributions give at least
\[
 2(n-1)-t'
\]
zero-before-one pairs. Each of the remaining $t-2$ suffix ones
pairs with at least the first zero. Thus
\[
 |\mu|\ge 2(n-1)-t'+t-2\ge2n-3,
\]
since $t'<t$.

Now suppose that $\mu$ begins with $01^t$ but
$r>\#_1(\mu)-2t-2$. The first suffix one paired with the prefix zeros,
and the prefix ones paired with the first zero, give $n-1$ pairs.
There are more than $(n-1)-t-2-r$ prefix zeros, so the last suffix
one contributes at least $(n-1)-(t+2)$ further pairs. Pairing the
remaining $t-2$ suffix ones with the first zero gives
\[
 |\mu|\ge 2n-4-t+(t-2)=2n-6.
\]
Both cases contradict $|\mu|=d\le2n-8$. Hence $t=1$.

\paragraph{The map $\operatorname{strip}$.}
First define 
$\operatorname{strip}(\epsilon,s)=(\varnothing,s) \in  \mathscr C$.

For a bad $\mu$ with at least $n-1$ $0$s, write $\mu=0V\,0^b1$, where $V$ does not end in $0$, and define
\[
\operatorname{strip}(\mu,s)=(0V,s) \in  \mathscr R. 
\]

For a bad $\mu$ with at most $n-2$ $0$s, note that $t=1$ implies it has at least two $1$s. Write $\mu=0^{a+1}1\,V\,0^b1$, where $V$ does not end in $0$, and define
\[
\operatorname{strip}(\mu,s)=(0\,V , s) \in  \mathscr C.
\]

\paragraph{The map $\operatorname{strip}^{-1}$.}
When $d=n-2$, define $\operatorname{strip}^{-1}( (\varnothing,s) \in\mathscr C)=(\epsilon,s)$. Otherwise, write $\rho=0V$.  

If $(\rho,s) \in \mathscr R$, define $ b = d-|\rho|-\#_0(\rho)$ and set 
\[
\operatorname{strip}^{-1}(\rho,s)=(0V0^b1,s).
\]

If $(\rho,s) \in \mathscr C$, set
\begin{align}
 a&=\left\lceil
       \frac{d-n-|\rho|+1}{\#_1(\rho)+1}
     \right\rceil,\label{eq:strip-inverse-ceiling}\\
 b&=d-|\rho|-\#_0(\rho)-1-(\#_1(\rho)+2)a.
 \label{eq:strip-inverse-size}
\end{align}
Define
\[
 \operatorname{strip}^{-1}(\rho,s)
   =\bigl(0^{a+1}1\,V\,0^b1,\ s\bigr).
\]

\begin{lemma}
\label{lem:strip}
The maps $\operatorname{strip}$ and $\operatorname{strip}^{-1}$
are inverse bijections between the excluded pairs and
$\mathscr R\sqcup\mathscr C$.
\end{lemma}
\begin{proof}

We now show that $\operatorname{strip}$ is well defined.
It is apparent that sending  $(\epsilon,s)\leftrightarrow(\varnothing,s)$ is
valid.
If $\#_0(\mu)\ge n-1$, removing $0^b1$ decreases size by $\#_0(\mu)$ and the number of $1$s by one.
Thus the resulting pair belongs to $\mathscr R$.

If $\#_0(\mu)\le n-2$, removing  $0^{a}1$ and $0^b1$ results in a loss of  $d-|\rho|= \#_0(\mu)+1+a(\#_1(\mu)-1)$ zero-before-one pairs.
If $a>0$, this is at least $\#_0(\mu)+\#_1(\mu)\ge n$.
If $a=0$, then $t=1$ implies that badness occurs due to the failure 
\[
 r=\#_0(\mu)+(\#_1(\mu)-t)-(n-1)>(\#_1(\mu)-2t)-2,
\]
so $\#_0(\mu)\ge n-3$.
Thus $d-|\rho|\ge n-2$ in either case. Also,
$\#_1(\rho)+2=\#_1(\mu)\le s$, so $(\rho,s)\in\mathscr C$.

Next we show that $\operatorname{strip}^{-1}$ is well defined, mapping
$\mathscr R$ to the excluded pairs with at least $n-1$ zeros and
$\mathscr C$ to those with at most $n-2$ zeros.
For $(\rho,s)\in\mathscr R$, first note $b\ge d-2|\rho|-1\ge2n-d-3\ge0$. Next, the size bound on $\rho$ means that the
reconstructed partition has at least $n-1$ zeros before its final $1$,
so $t\ge1$. Moreover,
\[
 r\ge \#_1(\mu)-t>\#_1(\mu)-2t-2,
\]
and hence the reconstructed partition is bad (and lies in the pre-image subset with at least $n-1$ $0$s).

For $(\rho,s)\in\mathscr C$ with $d>n-2$, write
$\rho=0V$ and
\[
 \mu=0^{a+1}1\,V\,0^b1
\]
for the prescribed image. It has $\#_1(\rho)+2$ ones, so having
at least $n$ entries and at most $n-2$ zeros amounts to
\[
 \#_0(\rho)+a+b\in[n-\#_1(\rho)-2,n-2].
\]

To get a partition of size $d$, equation~\eqref{eq:strip-inverse-size} must hold.  
Substituting  \eqref{eq:strip-inverse-size} into the interval expression above gives
\[
 d-|\rho|-1-a(\#_1(\rho)+1)
 \in[n-\#_1(\rho)-2,n-2],
\]
or equivalently,
\[
 a(\#_1(\rho)+1)
 \in[d-n-|\rho|+1,d-n-|\rho|+1+\#_1(\rho)].
\]
This interval contains exactly one multiple of $\#_1(\rho)+1$,
the one given by \eqref{eq:strip-inverse-ceiling}. Thus the prescribed values
are the unique values of $a$ and $b$ such that $0^{a+1}1\,V\,0^b1$
has at least $n$ entries, has at most $n-2$ zeros, and has size $d$. 
To show that this candidate actually lies in the target, it remains to check that $b \geq 0$
and that it is bad.

Equation~\eqref{eq:strip-inverse-ceiling} gives
$0\le a\le d-|\rho|-n+2$. Hence, using the hook estimate $\#_0(\rho)+ \#_1(\rho)\le|\rho|+1$, we have
\[
 \#_0(\rho)+a
 \le (|\rho|-\#_1(\rho)+1)+(d-|\rho|-n+2)
 \le n-\#_1(\rho)-2,
\]
where the last inequality follows from the assumption $d \leq 2n-5$. 
Thus the lower bound on $\#_0(\rho)+a+b$ gives $b\ge0$.

The reconstructed word has at least $n$ entries, so $t\ge1$.
If $a>0$, it begins with $00$ and is bad. If $a=0$,
\eqref{eq:strip-inverse-ceiling} gives $d-|\rho|\le n-1$, while
$(\rho,s)\in\mathscr C$ gives $d-|\rho|\ge n-2$. Hence
$d-|\rho|\in\{n-2,n-1\}$, and \eqref{eq:strip-inverse-size} gives
\[
 \#_0(\mu)=\#_0(\rho)+b=d-|\rho|-1\in\{n-3,n-2\}.
\]
Therefore
\[
 r=\#_0(\mu)+(\#_1(\mu)-t)-(n-1)
 \ge\#_1(\mu)-t-2
 >\#_1(\mu)-2t-2.
\]
Thus $\mu$ is bad in either case.

After applying $\operatorname{strip}$ in the $\mathscr C$ case,
putting back the deleted ones and adding back zeros for some values
of $a,b$ must return the original word. But we have shown that any
values other than the prescribed $a,b$ land outside the required
subset of the domain. Therefore the prescribed values must return
the original word. In the $\mathscr R$ case, size determines how
many zeros to add back. Thus $\operatorname{strip}^{-1}$ inverts
$\operatorname{strip}$.

Applying $\operatorname{strip}^{-1}$ and then $\operatorname{strip}$
clearly returns the original pair.
\end{proof}

\subsection{The high side}
To construct the high-side partners of the excluded pairs, we now pass from binary words to ordinary partition notation. Parts will be row lengths, and $\ell(\rho)$ will denote the number of parts. Write $\Delta_N=(N-1,\ldots,1)$. A partition $\gamma\subseteq\Delta_N$ represents the Dyck sequence
\[
 \bigl(i-1-\gamma^{\top}_{N+1-i}\bigr)_{i=1}^N,
\]
with parts zero-padded. Define its co-area to be the size of the embedded partition, $|\gamma|$. Replacing $\Delta_N$ with $\Delta_{N+1}$ acts on the sequence by prepending a  zero and raising the old entries by one, preserving dinv. Hence dinv is defined for every partition, independently of the containing staircase, and we put $\defc(\gamma)=|\gamma|-\dinv(\gamma)$. A partition represents a path of semilength $n$ precisely when it lies in $\Delta_n$; call it \emph{bad} otherwise.

For $s\ge\ell(\rho)$, define $\operatorname{pad}(\rho,s)$ to adjoin a column of height $s$: i.e., add one to each part and append $s-\ell(\rho)$ ones. We use the size-preserving length-to-dinv bijection of Loehr and Warrington~\cite[Sections~2 and~6.1]{LoehrWarrington09}. Denote it by $\operatorname{LW}$: increase their slope parameter from zero to one, including the switch at one, and then conjugate the output to agree with our diagram convention. Thus
\[
 |\operatorname{LW}(\rho)|=|\rho|,
 \qquad \dinv(\operatorname{LW}(\rho))=\ell(\rho).
\]

We will apply this bijection after adjoining a column, and then adjoin a long row or column. The following lemma controls these final additions and shows how they can be recovered from a bad partition.

\begin{lemma}
\label{lem:clearance}
If $|\tau|\le3n-16$ and $\defc(\tau)\le n-6$, then
$\tau\subseteq\Delta_{n-2}$. If a bad $\gamma$ satisfies
$|\gamma|\le4n-16$ and $\defc(\gamma)\le2n-8$, then exactly one
of its first row and first column has length at least $n$,
and deleting it leaves a partition in $\Delta_{n-2}$
 and decreases dinv by one or two, respectively.
Adjoining a row or column of length at least $n$ to a partition
in $\Delta_{n-2}$ gives a bad partition and increases dinv by
one or two, respectively.
\end{lemma}
\begin{proof}
For the first statement, suppose $\tau$ does not lie in $\Delta_{n-2}$ and 
choose $T \geq n-1$ minimal so that $\tau$ lies in
$\Delta_T$. Its Dyck sequence, $D$, has a second $0$: otherwise
the embedded partition never touches the diagonal and would fit
in $\Delta_{T-1}$. We count deficit pairs as in the proof of
Lemma~\ref{lem:wrap}.
If a $3$ occurs in $D$, pair each entry at least two
with the second $0$, and each other nonfirst $0$ or $1$ with the
first $3$. This gives $T-3$ distinct pairs; only the first two
$0$s and the first $1$ are not counted. Thus
$\defc(\tau)\ge T-3\ge n-4$, a contradiction. Without a $3$,
the second $0$ makes the area at most $2T-5$, whence
\[
 |\tau|\ge\binom T2-2T+5>3T-13\ge3n-16,
\]
also a contradiction. This proves the first assertion.

Now let $\gamma$ be as in the second assertion. Choose $T$ minimal
so that it lies in $\Delta_T$, so $T\ge n+1$. If no $4$ occurs in the resulting Dyck sequence, $D$,
then since two $0$s must occur by minimality of $T$, the area is
at most $3T-9$, and
\[
 |\gamma|\ge\binom T2-3T+9>4T-20\ge4n-16.
\]
Thus a $4$ occurs, while $\defc(\gamma)\le2n-8\le2T-10$. 

Now if there are at least three $0$s, consider the second and third
$0$s. Each entry at least two pairs with both, while every other
nonfirst $0$ or $1$ pairs with the first $3$ and $4$. This gives
$2(T-4)>2T-10$ distinct pairs, a contradiction. Thus there
are exactly two $0$s.

Next suppose there are at least two $1$s. We will use the second
$0$ and second $1$. Each entry at least $3$ pairs with both;
every nonfirst $2$ pairs with the second $0$ and the first $4$;
and the other nonfirst $0$s and $1$s pair with the first $3$
and $4$. The first $2$ also pairs with the second zero. These
are $2(T-5)+1>2T-10$ distinct pairs, again a contradiction.

Hence $D$ has exactly two $0$s and one $1$ (if it had no $1$ it would have to be all $0$s). The second zero
must be the second or last entry: any other position would force
two $1$s. Deleting that zero removes a first row or first column
of length $T-1\ge n$, respectively. The remaining sequence has
a unique $0$ and a unique $1$, so the dinv removed is one in
the row case and two in the column case. The remaining partition
satisfies
\[
 |\tau|\le4n-16-(T-1)\le3n-16,\qquad
 \defc(\tau)\le2n-8-(T-1)+2\le n-6.
\]
The first part of the lemma therefore puts $\tau$ in
$\Delta_{n-2}$. 

For the final statement, let $\tau\subseteq\Delta_{n-2}$, and let $L\ge n$ be
the length of the row or column to be adjoined. View $\tau$ in
$\Delta_L$, with Dyck sequence $D$. Since
$\tau\subseteq\Delta_{L-2}$, $D$ begins with $01$ and every
remaining entry is at least $2$. After adjoining the row or
column, view the resulting partition in $\Delta_{L+1}$.
Its sequence is $0:D$ in the row case and $D:0$ in the column
case. The added $0$ forms a dinv pair only with the original
$0$ in the row case, and with the original $0$ and $1$ in
the column case. Thus dinv increases by one or two, respectively.
In both cases, the added row or column has length $L\ge n$,
so the resulting partition lies outside $\Delta_n$.
\end{proof}

\paragraph{The map $\operatorname{build}$.}
For $(\rho,s)\in\mathscr R$, put
$\tau=\operatorname{LW}(\operatorname{pad}(\rho,s-1))$ and adjoin
a row of length $d+1-|\rho|$:
\[
 \operatorname{build}(\rho,s)=(d+1-|\rho|,\tau).
\]
For $(\rho,s)\in\mathscr C$, put
$\tau=\operatorname{LW}(\operatorname{pad}(\rho,s-2))$ and adjoin
a column of height $d+2-|\rho|$:
\[
 \operatorname{build}(\rho,s)=\operatorname{pad}(\tau,d+2-|\rho|).
\]

\paragraph{The map $\operatorname{build}^{-1}$.}
For a bad high-side partition $\gamma$ of deficit $d$, put
$s=\dinv(\gamma)$. Remove its first row if this has length at
least $n$, and its first column otherwise. Apply
$\operatorname{LW}^{-1}$ to the remainder and remove its first
column to obtain $\rho$. Return $(\rho,s)$ in $\mathscr R$ if
a row was removed from $\gamma$, and in $\mathscr C$ if a column
was removed.

\begin{lemma}
\label{lem:build}
The maps $\operatorname{build}$ and $\operatorname{build}^{-1}$
are inverse bijections between $\mathscr R\sqcup\mathscr C$
and the bad high-side partitions of deficit $d$.
\end{lemma}
\begin{proof}
We first show that $\operatorname{build}$ is well defined.
In both cases, the defining bounds on $\mathscr R$ and
$\mathscr C$ give
\[
 \defc(\tau)=|\rho|\le n-6,\qquad |\tau|\le3n-16.
\]
The row or column to be adjoined in the last step has length at least $n$.
Lemma~\ref{lem:clearance} therefore makes the construction valid
and shows that the output is bad, with co-area $d+s$ and dinv $s$.

For $\operatorname{build}^{-1}$, take a bad partition $\gamma$
with co-area $d+s$ and dinv $s\le d$. Since
$|\gamma|=d+s\le4n-16$, Lemma~\ref{lem:clearance} identifies its
unique row or column of length $L\ge n$. Removing it lowers dinv
by one or two, respectively. After applying
$\operatorname{LW}^{-1}$ to the remainder and removing its first
column, we therefore obtain
\[
\begin{aligned}
 |\rho|&=d-L+1,\qquad \ell(\rho)\le s-1
 &&\text{in the row case},\\
 |\rho|&=d-L+2,\qquad \ell(\rho)\le s-2
 &&\text{in the column case}.
\end{aligned}
\]
Since $L\ge n$, these are exactly the bounds for $\mathscr R$
and $\mathscr C$, so $\operatorname{build}^{-1}$ is well defined.

Suppose $\operatorname{build}(\rho,s)=\gamma$.  By Lemma \ref{lem:clearance}  and well-definedness 
exactly one of its first row and first column has length at least $n$, so we know $\operatorname{build}^{-1}$ removes the added row/column. Since $\operatorname{LW}^{-1}$ inverts $\operatorname{LW}$, applying it and
removing the first column from the partition returns the starting pair $(\rho,\dinv(\gamma))$.  Further we know
whether the  pair came from $\mathscr R$ or $\mathscr C$ based on whether a row or column was removed in the first step. 

Suppose that $\operatorname{build}^{-1}(\gamma) =(\rho,s)$ in  $\mathscr R$ (resp.  $\mathscr C$).  We know how many boxes were removed by $\operatorname{pad}^{-1}$ by $s$ and the pair's membership.  We also know $\operatorname{LW}$ is invertible, so after re-adding those boxes we can recover the partition that existed after removal of a row/column.  We know whether a row or column was removed by membership and how long it must be by the fact the original partition was of size $d+s$, so we can recover $\gamma$.

Since these steps are exactly those prescribed by $\operatorname{build}^{-1}$ (resp. $\operatorname{build}$), the maps are mutual inverses.
\end{proof}

\subsection{Finite cancellation and the extension}
Adjoining a first column bijects pairs $(\mu,s)$ with $|\mu|=d$ and $\ell(\mu)\le s\le d$ with partitions of size $d+s$ and length $s$. Consequently,
\[
 (\mu,s)\longmapsto\operatorname{LW}(\operatorname{pad}(\mu,s))
\]
is a bijection to all partitions of co-area $d+s$ and dinv $s$ (with no requirement to lie in $\Delta_n$). The map $\operatorname{build}\circ\operatorname{strip}$ bijects the excluded pairs with those outside $\Delta_n$. Finite cancellation~\cite{Doyle19}, therefore,  gives a bijection from the
nonexcluded pairs $(\mu,s)$ to the high-side paths of dinv $s$.
Combining this with the low-side parametrization by area $s$ and $\operatorname{wrap}^{-1}$ gives
a bijection $\sigma$ from the low side to the high side
interchanging area and dinv.

Extend $\mathcal I$ by $\sigma$ on paths with $\area<d$ and by $\sigma^{-1}$ on paths with $\dinv<d$, retaining its previous definition elsewhere.

\begin{theorem}
\label{thm:extension}
For $n>4$, the involution $\mathcal I$ extends to all Dyck paths satisfying $\defc\le2n-8$, preserving deficit and interchanging area and dinv.
\end{theorem}
\begin{proof}
Since $M>3d$, the two added regions are disjoint and their restrictions are inverse bijections. The remaining paths are covered by Theorem~\ref{thm:area-dinv-involution}.
\end{proof}

For deficit $d \leq 2n-8$, we have shown that the low (resp. high) side Dyck paths are in bijection
with the nonexcluded pairs $(\mu \vdash d, s)$ where $s$ corresponds to $\area$ (resp. $\dinv$). 
Equivalently, the number of paths of deficit $d$ and $\area$ (resp. $\dinv$) equal to $s$ is the number of \emph{all}
pairs $\mu \vdash d, s \in [\ell(\mu),d]$ minus the number of pairs $(\rho,s)$ where $\rho$ has size at most $d+1-n$ and $s \in [\ell(\rho)+1,d]$ minus   the number of pairs $(\rho,s)$ where $\rho$ has size at most $d+2-n$ and $s \in [\ell(\rho)+2,d]$.  
It follows we can abstractly decompose the low (resp. high) side paths into strings with $\area$ (resp. $\dinv$) starting on $c$ and ending on $d$ such that the number of such strings is $p_{c}(d)-P_{c-1}(d-n+1)-P_{c-2}(d-n+2)$, where $p_k(m)$ indicates the number of partitions of $m$ into $k$ parts and $P_k(m)$ indicates the number of partitions with size at most $m$ with $k$ parts.  
This along with the flatness of the middle 
coefficients gives a completely combinatorial proof of the following.

\begin{corollary}\label{cor:form}
Write $\operatorname{sym}(a,b)= q^at^b+q^{a+1}t^{b-1}+\cdots+q^{b-1}t^{a+1}+q^bt^a$ and $M = \binom n2$
and retain the partition notation above. 
The $q,t$-Catalan polynomial corresponding to paths of deficit at most $2n-8$ can be written:
\[
\left.C_n(q,t)\right|_{\deg_{q,t}\ge M-2n+8}
=
\sum_{d=0}^{2n-8}\sum_{c=0}^{d}
\left(
p_c(d)-P_{c-1}(d-n+1)-P_{c-2}(d-n+2)
\right)
\operatorname{sym}(c,M-d-c).
\]

\end{corollary}

\begin{remark}
Summing across the strings, Corollary~\ref{cor:form} gives an explicit formula for the flat middle coefficients  not visibly equivalent to that of Theorem 5.35 of \cite{Hawkes26}. 
Interestingly, a formula that \emph{is} visibly equivalent to this sum was proposed by Lee and Li~\cite[Conjecture~4]{LeeLi11} much earlier. 
\end{remark}
  \begin{remark}
It is possible the Section~5 constructions can be extended significantly beyond
deficit $2n-8$.  
  In fact, with only minor modifications, the current construction extends to deficit
  $2n-7$: in addition to the exceptional singleton string
  $0^{n-1}1$, this only requires the additional low-side string
  $0^{n-3}101\longrightarrow 0^{n-3}120
  \longrightarrow 01\,0^{n-4}12$,
 and an additional family of bad high-side
  partitions whose first two columns have lengths $n$ and $n-2$.
  \end{remark}

\section{Further results}
\label{sec:theorems}
In this section we give three additional results (Theorems~\ref{thm:rational-flat-middle}, \ref{thm:dual-dyck-e-positive}, and~\ref{thm:three-slope-partition-formula}) related to the flat middle phenomenon, Dyck symmetric functions, and string decompositions for $q,t$-Catalan polynomials.  We have obtained proofs of these facts but because we believe these proofs are not particularly difficult or insightful we include only sketches.   Complete proofs can be reconstructed from these sketches\footnote{ For amusement purposes, we have tested that a GPT-6 class model can reliably reproduce a valid (albeit not particularly well written) proof of each result given a copy of this paper and prompted:
\emph{ ``For each result listed in Section~\ref{sec:theorems}, follow the proof sketch to obtain a complete proof of the theorem that uses no computer checks. Write using the terminology, notation and style of the pasted paper and its citations; new terminology you introduce should be minimized and must be explicitly defined.''}} and are available in~\cite{FurtherProofs}. 

\subsection{Flat middle for rational Dyck paths}
\label{subsec:rational-flat-middle}

Let $r,s\geq2$ be coprime integers, and let $\mathcal D_{r/s}$
be the set of lattice paths from $(0,0)$ to $(s,r)$ that stay below the line of slope $r/s$.
Following~\cite[Section~9]{Hawkes24}, write such a path in
\emph{step coordinates} as $x=(x_1,\ldots,x_s)$, meaning that
its step word is $EN^{x_1}EN^{x_2}\cdots EN^{x_s}$.
Thus the $x_i$ are nonnegative integers satisfying
\[
 \sum_{k=1}^{s}x_k=r,
 \qquad
 \sum_{k=1}^{i}x_k\leq
 \left\lfloor\frac{ri}s\right\rfloor
 \quad(1\leq i\leq s-1).
\]
Next, set
\[
 M_{r/s}=\sum_{k=1}^{s-1}\left\lfloor\frac{rk}{s}\right\rfloor
        =\frac{(r-1)(s-1)}2,
 \qquad
 \area(x)=M_{r/s}-\sum_{k=1}^{s}(s-k)x_k.
\]
Thus $M_{r/s}$ counts the whole lattice boxes below the diagonal,
and $\area(x)$ counts those between the path and the diagonal.

\begin{definition}[Rational deficit]
Following~\cite[Section~9]{Hawkes24}, for $2\leq i\leq j\leq s-1$ set
\[
 \beta_{ij}(x)=\sum_{k=i}^{j}x_k-\frac rs(j-i+1),
 \qquad
 \gamma_{ij}(x)=
 \begin{cases}
  \min\{x_{i-1},\lfloor|\beta_{ij}(x)|\rfloor\},
       &\beta_{ij}(x)<0,\\[2pt]
  \min\{x_i,\lfloor|\beta_{ij}(x)|\rfloor\},
       &\beta_{ij}(x)>0.
 \end{cases}
\]
The \emph{rational deficit} (referred to as $\operatorname{degr}$
in~\cite{Hawkes24}) is
\[
 \defc(x)=\sum_{2\leq i\leq j\leq s-1}\gamma_{ij}(x).
\]
Coprimality ensures that none of the indicated $\beta_{ij}$ is zero.
\end{definition}

\begin{theorem}
\label{thm:rational-flat-middle}
Let $r,s$ be coprime positive integers, and fix an integer
$d\geq0$ with $3d\leq M_{r/s}$.
Then
\[
 N_{r/s,d}(a):=
 \#\{x\in\mathcal D_{r/s}:\defc(x)=d,\ \area(x)=a\}
\]
is independent of the integer $a$ throughout
\[
 d\leq a\leq M_{r/s}-2d.
\]
\end{theorem}

\begin{example}
Set $r=7$, $s=5$ and $\defc =3$. Here $M_{7/5}=12$, so the proposed flat-middle interval is $[3,6]$.
The following table lists all the paths of deficit $3$ in this
interval. 
\begin{center}
\setlength{\tabcolsep}{5pt}
\begin{tabular}{@{\hspace{16mm}}ccc@{}}
\toprule
\makebox[0pt][r]{$\area=3$\hspace{3mm}}%
\RationalDyckCell{7}{5}{0,1,3,0,3} &
\RationalDyckCell{7}{5}{1,1,0,2,3} &
\RationalDyckCell{7}{5}{1,1,1,0,4} \\[3pt]
\midrule
\makebox[0pt][r]{$\area=4$\hspace{3mm}}%
\RationalDyckCell{7}{5}{0,0,3,2,2} &
\RationalDyckCell{7}{5}{1,0,1,2,3} &
\RationalDyckCell{7}{5}{1,0,2,0,4} \\[3pt]
\midrule
\makebox[0pt][r]{$\area=5$\hspace{3mm}}%
\RationalDyckCell{7}{5}{0,0,2,3,2} &
\RationalDyckCell{7}{5}{0,1,0,4,2} &
\RationalDyckCell{7}{5}{1,0,1,1,4} \\[3pt]
\midrule
\makebox[0pt][r]{$\area=6$\hspace{3mm}}%
\RationalDyckCell{7}{5}{0,0,1,4,2} &
\RationalDyckCell{7}{5}{0,2,0,0,5} &
\RationalDyckCell{7}{5}{1,0,0,2,4} \\[3pt]
\bottomrule
\end{tabular}
\end{center}
Thus, the theorem holds in this example. 
\end{example}

\paragraph{Proof sketch.}
Use the results of Ma and Mellit to express the fixed-deficit area
counts as dimensions in the associated graded vector space for the ideal of
positive-degree invariants~\cite[Theorem~4.6 and Proposition~7.2]{MaCuspidal}
\cite[Theorem~5.8]{Mellit}.
GORS's modified filtration equips each fixed-deficit piece with an
$\mathfrak{sl}_2$ action~\cite[Section~4.2, Proposition~4.5]{GORS}.
The key is to show that every irreducible string starts at area at
most $d$.
To bound these starting areas, quotient by the quadratic power sum $p_2$.
Since $p_2$ acts as a raising operator, the quotient detects the starts
of the irreducible strings.
Every remaining power-sum generator increases area by one and deficit
by at least one, which gives the required bound.
Weight symmetry then shows that each string covers the entire proposed
middle interval and contributes equally at every area in it.

\subsection{Elementary positivity of dual Dyck symmetric functions}
\label{subsec:dual-dyck-e-positive}

We recall the dual Dyck symmetric functions of~\cite{Hawkes26}.
Let $X$ be a finite multiset of integers. A \emph{dual Dyck
factorization} of $X$ is a sequence
$\mathcal F=(F_1,F_2,\ldots)$ of integer words, all but finitely
many empty, whose combined content is $X$ and such that successive
entries in each factor increase by at least two.

Writing $w=w_1\cdots w_n$ for the concatenation of the factors, put
\[
 \di(\mathcal F)
 =\#\{(i,j):1\leq i<j\leq n,\ w_i=w_j+1\},
 \qquad
 \mathbf x^{\operatorname{wt}(\mathcal F)}
 =\prod_{k\geq1}x_k^{|F_k|}.
\]
The \emph{dual Dyck symmetric function} is
\[
 \DSdual(X,d;\mathbf x)
 =\sum_{\substack{\mathcal F\text{ a dual Dyck factorization of }X\\
                   \di(\mathcal F)=d}}
       \mathbf x^{\operatorname{wt}(\mathcal F)}.
\]

These functions are symmetric and
Schur-positive~\cite[Theorem~3.18 and Corollary~3.19]{Hawkes26}.

\begin{theorem}
\label{thm:dual-dyck-e-positive}
For every finite multiset $X$ of integers and every integer $d\geq0$,
the function $\DSdual(X,d;\mathbf x)$ is $e$-positive. More precisely,
\[
 \DSdual(X,d;\mathbf x)
 =\sum_{\lambda\vdash |X|}c_{X,d,\lambda}\,e_\lambda(\mathbf x),
 \qquad c_{X,d,\lambda}\in\mathbb Z_{\geq0}.
\]
\end{theorem}

\begin{example}
\label{ex:dual-dyck-six-letters}
Take $X=\{1,2,2,3,4,5\}$ and $\di=3$, and specialize to the three
variables $\mathbf x_3=(x_1,x_2,x_3)$.
The factorizations of dominant weight are precisely the following five:
\[
\begin{array}{c|l|c}
 \text{Weight}&\text{Factorizations with }\di=3&\text{Number}\\
 \hline
 (3,2,1)&135\mid24\mid2&1\\[4pt]
 (2,2,2)&
 \begin{aligned}
  &13\mid25\mid24,\quad24\mid13\mid25,\\
  &24\mid25\mid13,\quad25\mid13\mid24
 \end{aligned}
 &4
\end{array}
\]
By symmetry, these
counts determine the entire three-variable monomial expansion:
\[
 \DSdual(X,3;\mathbf x_3)
 =m_{321}(\mathbf x_3)+4m_{222}(\mathbf x_3).
\]
In three variables,
\[
 e_{33}=m_{222},
 \qquad
 e_{321}=m_{321}+3m_{222}.
\]
Consequently,
\[
 \boxed{\DSdual(\{1,2,2,3,4,5\},3;\mathbf x_3)
        =e_{321}(\mathbf x_3)+e_{33}(\mathbf x_3).}
\]
\end{example}

\paragraph{Proof sketch.}
Let $\di^*$ refer to the $\di$ statistic  of the factorization with its factors ordered in reverse. It suffices to prove $e$-positivity for fixed $\di^*$.  Next view a dual factorization as a weighted coloring~\cite[Definition~4.1]{KiemLee24} of the graph whose
vertices are the distinct integers in $X$ and whose edges join
consecutive integers. Each vertex has multiplicity equal to the corresponding multiplicity in $X$
and is colored by all indices corresponding to a factor containing its value. Under this correspondence $\di^*$ is mapped to the ascent statistic
of~\cite[Definition~4.3]{KiemLee24}.  To prove the $e$-positivity of the weighted chromatic quasisymmetric functions (which would answer Question~4.9 of \cite{KiemLee24})
extend the coloring identity
in~\cite[Proposition~2.3]{OpenAI26} by extracting the quantum-torus
coefficient that records the weight of each vertex, rather than a
squarefree coefficient.
After embedding the graph in a triangular chart using
\cite[Lemma~5.1]{OpenAI26}, apply the monomial-to-theta identity and
theta positivity~\cite[Lemma~5.7 and Theorem~5.8]{OpenAI26} to obtain an
elementary expansion of $\sum_{\mathcal F}q^{\di^*(\mathcal F)}\mathbf x^{\operatorname{wt}(\mathcal F)}$ with nonnegative
coefficients in $q$.
Extract the coefficient of $q^d$ to obtain $e$-positivity for fixed $\di^*=d$.

\subsection{Dyck paths with deficit at most \texorpdfstring{$3n-15$}{3n-15}}
\label{subsec:three-slope-partition-formula}

Let $\mathcal F_n$ denote the full Dyck skeletons of length $n$, and define a set of uppers
\[
 \mathcal U_n=\left\{(P,\varnothing,(2)):
 \begin{array}{l}
 |P|=n-1,\quad \max(P)=2,\quad P[-1]=0
 \end{array}\right\}.
\]
Define an exceptional subset of full Dyck skeletons by
\[
 \mathcal E_n=
 \{0010^{n-4}1\}\cup\{001120^{n-6}1\}\cup\{001B12\}\cup\{00112B12\},
\]
where  $B$ is a binary word, and only words of
length $n$ are included.
For $a \leq d$, define $\mathcal N_n(d,a)$ to be
the number of uppers in $\mathcal U_n$ of deficit $d$ and area $a$,
plus the number of full Dyck skeletons in $\mathcal F_n\setminus\mathcal E_n$
of deficit $d$ and area $a$. The deficit and area of an upper are those of its
associated Dyck path under $\psi^{-1}$.

\begin{theorem}
\label{thm:three-slope-partition-formula}
Write $\operatorname{sym}(a,b)= q^at^b+q^{a+1}t^{b-1}+\cdots+q^{b-1}t^{a+1}+q^bt^a$ and $M = \binom n2$. Then,
\[
 \left.C_n(q,t)\right|_{\deg_{q,t}\ge M-3n+15}
 =\sum_{d=0}^{3n-15}\sum_{a=0}^{d}
 \mathcal N_n(d,a)\operatorname{sym}(a,M-d-a).
\]
\end{theorem}

\begin{example}
The full Dyck skeletons of length $9$, deficit $12$, and area $9$ are
\[
 001112121,\qquad 001121112,\qquad 001122210.
\]
Of these, $001121112$ belongs to $\mathcal E_9$.
There is also exactly one upper in $\mathcal U_9$ with these statistics:
\[
 (00112110,\varnothing,(2)),
\]
whose associated Dyck path is $001123110$. Consequently,
\[
 \mathcal N_9(12,9)=1+(3-1)=3.
\]
This gives the summand $3\operatorname{sym}(9,15)$ in the stated
homogeneous component of degree $24$.
Direct enumeration gives $36$ deficit-$12$ Dyck paths of area $8$
and $39$ of area $9$, whose difference is indeed $3$.
\end{example}

\paragraph{Proof sketch.}

Consider deficits $d$ at most $3n-15$. Define a set of lowers $\mathcal L_n$ that is the union of the following forms: $(0^{n-1},(1),\emptyset)$, $(0^{n-3},(12),(0))$, $(00B01,(2),\emptyset)$, $(0011B11,(2),\emptyset)$ where $B$ is as before. Note that the elements of $\mathcal L_n$ and $\mathcal U_n$ each arise from a key which supports no other lowers or uppers. In light of the deficit preserving, area raising bijection from $\mathcal L_n$ to $\mathcal E_n$:
\begin{align*}
(0^{n-1},(1),\emptyset)&\longmapsto001\,0^{n-4}1,\\
(0^{n-3},(12),(0))&\longmapsto00112\,0^{n-6}1,\\
(00B01,(2),\emptyset)&\longmapsto001\,B\,12,\\
(0011B11,(2),\emptyset)&\longmapsto00112\,B\,12,
\end{align*}
it is enough to show that no other unbalanced keys with area less than $d$ exist to establish that $\mathcal N_n(d,a)$ counts the difference between the number of deficit-$d$ Dyck paths of areas $a$ and $a-1$ for $a\le d$. To accomplish the former, first prove that there are no such keys for $k\ge2$. 

Indeed if there were such a key it would support a lower or an upper. One can then show such a key $(P,X,h)$ must have $|P[-1]-(\operatorname{max}(P)-1)|\le2$. This allows one to generalize the telescoping $pTs$ argument by replacing the endpoints $0$ and $-1$ by $P[-1]$ and $\operatorname{max}(P)-1$, respectively. Then one can show that whenever a generalized $pTs$ object exists on a key supporting an upper or lower, $h\le k^2$ and at most $k+1$ entries of $X$ lie in $\{v-1,v\}$ for each $v$, with the latter bound improving to $k$ when $v\le\min(P[-1]+1,\max(P))$. These bounds on $h$ and the multiplicities in $X$, together with the formulas for area and deficit of a key $K=(P,X,h)$, give $\area(K)+\defc(K)-(6n-30)\ge(k-6)(|P|+k-4)$, which contradicts $\area(K)<\defc(K)\le3n-15$ whenever $k\ge6$. 
For $2\le k\le5$, use the same bounds together with the formulas for area and deficit to derive contradictions to $d\le3n-15$ or $\area(K)+d<6n-30$, treating separately the cases $P[-1]=\max(P)$ and $P[-1]<\max(P)$.
 Thus no such generalized $pTs$ objects can exist for $k\ge2$ on a key that also supports a lower or upper. Next one can prove that for $k=0$ and $k=1$ the only unbalanced keys with area less than $d$ are those that $\mathcal L_n$ and $\mathcal U_n$ arise from. Finally, the theorem then follows from Theorem~\ref{thm:flat-middle} and the known symmetry of the $q,t$-Catalan polynomial.
\section*{Acknowledgments}
The author thanks Nicholas Loehr and Kyungyong Lee for useful conversations.  Codex was used for coding assistance, conjecture experimentation,  assistance formatting the paper in \LaTeX{}, and for assistance with the linear programming in the appendix.

\appendix
\numberwithin{equation}{section}
\section{Proof of the inequality}
\label{app:inequality}
\theoremstyle{plain}
\newtheorem{proposition}[definition]{Proposition}
\providecommand{\ind}[1]{\mathbf1_{\{#1\}}}
\providecommand{\loss}{\operatorname{loss}}
\providecommand{\out}{\operatorname{out}}

In this appendix we prove Lemma~\ref{ineq}. We begin by showing that the lemma holds for all multisets $X$ except those whose entries come from a finite set of values: $\{-4,-3,\ldots,10\}$ in Section~\ref{app:outer-range}.  In Section~\ref{app:remaining-alphabet} we consider these  bounded multisets and introduce a notion of frames and fillings.  A \emph{factorization} or \emph{original-row} frame and a \emph{filling} bound the $\dinv$ of a $pTs$ or $22$-$pTs$ object and are useful because the number of frames to be considered is finite and fillings can be described with a finite set of  variables.  These are the two important reductions that allow the linearization of Section~\ref{app:linearization} and the linear programming checks of Section~\ref{app:exact-checks} to be able to establish all the prerequisites for the proof given in Section~\ref{app:completion}.

We consider the two reductions in Sections~\ref{app:outer-range} and~\ref{app:remaining-alphabet} to be the key insights into understanding why the lemma should be true. Beyond that, all further analysis has the feel of verification rather than edification. Therefore, although extensive case analysis can produce a completely analytic proof of the lemma, we believe that such a proof is no more enlightening and much more time consuming for a reader than the proof currently provided.

\subsection{The outer range}\label{app:outer-range}
\begin{proposition}\label{thm:main}
Let $X$ be a multiset of  $2k+1$ integers, where $k\ge2$, with minimum $\ell$
and maximum $H$. If $H\ge11$ or $\ell\le-5$, then
\begin{equation}\label{eq:main}
 2\min\{h_z(X),h_0(X)\}\le\beta(X).
\end{equation}
\end{proposition}
\begin{proof}
Assume the contrary so that there is an upper or lower and a $pTs$ object on $X$ both of  whose doubled dinv exceeds $\beta(X)$.  In particular, the existence of the former and the latter of these objects implies that  $\ell\le0$ and $H\ge1$, respectively.
Let $m_i$ be the number of occurrences of $i$ in $X$, and let $t$ be the
 smallest unique entry of the multiset obtained by adding a $0$ to the nonnegative entries of $X$, or $\infty$
if none exists. Set
\[
 w(i)=\begin{cases}i+2&i<0,\\0&i=0,\\i-2&i>0.\end{cases}
\]
It follows that $\beta=2k^2-1-\sum_iw(i)m_i.$ 
The existence of a simple lower/upper implies
\begin{equation}\label{eq:conditions}
 m_1,\ldots,m_H\ge1,\qquad
 \sum_{i>t}m_i\le k+1-t\quad(t<\infty).
\end{equation}
To see the first, note that for a lower or upper, the factor containing $H$ must contain also $H-1, \ldots,1$.  
To see the second, note that if $0<t<\infty$, all entries greater than $t$ lie in the same factor (along with $t,t-1,\ldots,1$). If $t=0$, all positive entries lie in a single factor.

Take our assumed $pTs$ object, and create a factorization composed of $p$, the tableau rows from bottom to top, and the reversal
of $s$. This factorization has the content and dinv of the $pTs$ object. Since the
tableau has at most $k-1$ rows if both outer factors are nonempty, and at
most $k$ otherwise, there are at most $k+1$ factors.
Applying dual Dyck insertion \cite[Theorem~3.18]{Hawkes26} to this factorization gives a pair of tableaux $(P,Q)$ whose common shape has at most $k+1$ rows and thus dominates the  shape $(2^k,1)$.  Now replace $Q$ with $Q'$,
the recording tableau of the same shape and weight
$(2^k,1)$  \cite[Proposition~5.1(b)]{Stembridge14}, and apply the inverse of the dual Dyck insertion bijection to $(P,Q')$.
The result is a factorization $F$ into $k$ factors of length 2 (pair factors) and one singleton with the same
content and dinv as the $pTs$ object.

Between two pair factors of $F$ there are at most two dinv pairs; between a pair factor
and the singleton there is at most one. The maximum possible dinv of $F$ is therefore
$2\binom{k}{2}+k=k^2$. Set $D=k^2-\dinv(F)$.

Suppose the singleton and $r\in\{1,2\}$ pair factors are \emph{reserved}. Denote
their entries by $\mathcal R$ and call the remaining $k-r$ pairs \emph{free}.
Set at most two marks on each reserved entry and at most one on each free
entry. A mark is \emph{far from} a free pair if its value is at distance
at least two from both entries. Write $o(G)$ for the number of marks far
from a free pair $G$ and let $o(G,J)$ count the marks in $J$ far from a free pair $G$. Then

\begin{equation}\label{eq:marked-budget}
\begin{aligned}
 \sum_{G\text{ free}}o(G)
 &=\sum_{G\text{ free}}\left(
       \sum_{\substack{J\text{ free}\\J\ne G}}o(G,J)
       +\sum_{R\text{ reserved}}o(G,R)\right)\\
 &\le2\sum_{\substack{G \prec J\\G,J\text{ free}}}z(G,J)
       +2\sum_{\substack{G\text{ free}\\R\text{ reserved}}}z(G,R)
 =2D-2\sum_{\substack{R\prec S\\R,S\text{ reserved}}}z(R,S)\le2D.
\end{aligned}
\end{equation}
where $\prec$ refers to preceding in the factorization order and where $z$ represents the lost potential
dinv between two factors in the order they appear in the factorization (two minus their dinv contribution if both are pair
factors, and one minus their contribution otherwise).

Moreover since $2\dinv(F)-\beta=1+\sum_{x\in X}w(x)-2D$, \eqref{eq:marked-budget} gives
\begin{equation}\label{eq:marked-sufficient}
 1+\sum_{x\in X}w(x)\le\sum_{G\text{ free}}o(G)
 \ \Longrightarrow\ 2\dinv(F)-\beta\le0.
\end{equation}

First suppose that $H\le3$ (and so $\ell\le-5$). Reserve the singleton factor and one pair such that
$\ell$ is reserved, denoting the reserved entries by $\ell,j,z$.
Set two marks on $\ell$. Since every entry has weight at most one,
a free pair $G$ far from $\ell$ has weight at most $2=o(G)$.
If the free pair $G$ is not far from $\ell$, it contains $\ell$ or $\ell+1$, so its weight is at most
$\ell+4\le0=o(G)$. Also $1+w(\ell)+w(j)+w(z)\le\ell+5\le0$.
This shows the hypothesis of \eqref{eq:marked-sufficient} holds, contradicting the assumption on the $pTs$ object.

Now assume $H\ge4$. Put $\Delta=2\dinv(F)-\beta(X)$.
By assumption, the doubled dinv of our $pTs$ object exceeds $\beta(X)$.
Since $F$ has the same dinv, $\Delta>0$.
A mark is \emph{near} an entry $x$ if the value of the entry carrying it
differs from $x$ by at most one; let $N(x)$ count these marks.
For a marking with $M$ marks obeying the preceding limits on free and
reserved entries, a free pair $(u,v)$ has at least $M-N(u)-N(v)$ far
marks. Thus \eqref{eq:marked-budget} gives
\begin{equation}\label{eq:marks}
 2D\ge (k-r)M-\sum_{x\text{ free}}N(x).
\end{equation}

\paragraph{The case $t\ge H-2$.}
Choose a sample $\mathcal B$ of elements of $F$ containing one zero,
two occurrences of each $1,\ldots,H-3$, and one occurrence of each
$H-2,H-1,H$. Start with one mark on each of the $2H-2$
sample entries and add two marks as follows.
\begin{itemize}
\item If $\ell\le-5$, choose an occurrence of $\ell$ and put two marks on it.  Reserve the singleton  of $F$ and one pair such that the chosen
occurrence of $\ell$ is reserved.
\item If $\ell>-5$, then $H\ge11$. Add an additional mark to a $1$ and $2$ in  $\mathcal B$.
Reserve two pairs and a singleton such that both of the doubly marked entries are reserved. 
\end{itemize}
Each marking has $2H$ marks and obeys the limits on free and reserved
entries. Assign each entry $x$ of $F$ \emph{credit}
\[
 \sigma(x)=H-w(x)-N(x).
\]

Since $\Delta=1+\sum_{x\in X}w(x)-2D$, \eqref{eq:marks} then gives
\begin{equation}\label{eq:uniform-credit}
 \Delta\le 1+(2r+1)H
       -\sum_{x\in X}\sigma(x)-\sum_{x\in\mathcal R}N(x).
\end{equation}

\medskip
\noindent\emph{Case $\ell\le-5$.}

First we compute $A:=\sum_{x\in\mathcal B}\sigma(x)$.   If $H=4$, the $0$ and both $1$s each contribute $1$ to the sum and all other elements $0$ so $A=3$. If $H>4$ then the credit of the elements $H-3, H-2,H-1,H$ is $0$, while for $i \in \{2,\ldots,H-4\}$ the credit is $H-i-4$. The credits for $0$ and $1$ are $H-3$ and $H-4$, respectively; thus, 
\[
 A=(H-3)+2(H-4)+2\sum_{i=2}^{H-4}(H-i-4)=(H-4)^2+3,
\]
where the formula $A=(H-4)^2+3$ also holds for $H=4$. 

The chosen occurrence of $\ell$ lies outside of $\mathcal{B}$, so its terms
in both sums in  \eqref{eq:uniform-credit} are not counted by $A$ and sum to 
$\sigma(\ell)+N(\ell)=H-w(\ell)$.
For either other reserved entry $x$, either $A$ counts its credit (in which case the reserved entry is in $\mathcal{B}$ and has $N(x) \ge 2$) or else  $\sigma(x)+N(x)=H-w(x) \ge 2$. Finally, any element that is not in $\mathcal{B}$ and not reserved has nonnegative credit so (with $r=1$) \eqref{eq:uniform-credit} gives
\[
 \Delta\le 1+3H-\bigl(A+H-w(\ell)+2\cdot2\bigr)
          =\ell+5-(H-5)^2\le0.
\]

\medskip
\noindent\emph{Case $\ell>-5$.}

First we compute $A:=\sum_{x\in\mathcal B}\sigma(x)$.
Each entry equal to one of $H-3,H-2,H-1,H$ has credit zero, while an
entry with value $i\in\{4,\ldots,H-4\}$ has credit $H-i-4$.
Entries with values $0,1,2,3$ have credits $H-4,H-6,H-8,H-8$,
respectively. Thus
\[
\begin{aligned}
 A&=(H-4)+2(H-6)+4(H-8)+2\sum_{i=4}^{H-4}(H-i-4)=(H-4)^2-8.
\end{aligned}
\]

The chosen occurrences of $1$ and $2$ lie in $\mathcal B$, so $A$
counts their credits but does not count the corresponding sum
$N(1)+N(2)=7+8=15$ in the last sum in \eqref{eq:uniform-credit}.
For each of the other three reserved entries $x$, either $A$ counts
its credit (in which case that entry is in $\mathcal B$ and has
$N(x)\ge2$), or else $\sigma(x)+N(x)=H-w(x)\ge2$.
Finally, any entry outside $\mathcal B$ that is not reserved has
nonnegative credit.
Thus, with $r=2$, \eqref{eq:uniform-credit} gives
\[
\begin{aligned}
 \Delta&\le 1+5H-\bigl(A+15+3\cdot2\bigr) =-(H-11)(H-2)-6<0.
\end{aligned}
\]
Both subcases contradict $\Delta>0$.

\paragraph{The remaining case $t\le H-3$.}
Set $c=H-t-2>0$. Choose a sample $\mathcal B$ containing one occurrence
of each $1,\ldots,H$. When $t>0$, also choose a zero and a second
occurrence of each $1,\ldots,t-1$.
Start with one mark on each of the $H+t$ sample entries and add two
marks as follows.
\begin{itemize}
\item If $\ell\le-5$, choose an occurrence of $\ell$ and put two marks
on it. Reserve the singleton and one pair such that this occurrence is
reserved.
\item If $\ell>-5$, then $H\ge11$. Choose sample occurrences of
$(a,a+1)=(1,2)$ when $t\ge6$, and of $(a,a+1)=(7,8)$ when $t\le5$.
Add one mark to each chosen occurrence. Reserve the singleton and two
pairs such that both chosen occurrences are reserved. 
\end{itemize}
Each marking has $M=H+t+2$ marks and obeys the limits on free
and reserved entries. Assign each entry $x$ of $F$ \emph{credit}
\begin{equation}\label{eq:credit}
 \sigma(x)=t+2+c\ind{x>t}-w(x)-N(x).
\end{equation}
Next, \eqref{eq:conditions} gives
\[
 \sum_{x\in X}\bigl(w(x)+N(x)+\sigma(x)\bigr)
 \le (2k+1)(t+2)+c(k+1-t).
\]
Since $\Delta=1+\sum_{x\in X}w(x)-2D$, \eqref{eq:marks} then gives
\begin{equation}\label{eq:global}
 \Delta\le 1+(2r+1)(t+2)+(r+1-t)c
       -\sum_{x\in X}\sigma(x)-\sum_{x\in\mathcal R}N(x).
\end{equation}

\medskip
\noindent\emph{Case $\ell\le-5$.}

First we compute $A:=\sum_{x\in\mathcal B}\sigma(x)$.  We claim it is equal to
\begin{equation}\label{eq:initial-count}
 A=\begin{cases}
 \binom{c+1}{2}+1,&t=0,\\[2pt]
 \binom{c+1}{2}+2,&t=1,\\[2pt]
 \binom{c+1}{2}+(t-2)^2+3,&t\ge2.
 \end{cases}
\end{equation}

Indeed for $t=0$ the credits across $\{1,2,\ldots,H-1,H\}$ are $\{c+1,c-1,c-2,\ldots, 2,1,0,0\}$.
For $t=1$ the credits across $\{0,1,2,\ldots,H-1,H\}$ are $\{1,1,c,c-1,\ldots, 2,1,0,0\}$.
 For $t = 2$,  note that $0$, $1$, and $2$ have respective credits $1$, $1$, and $0$, and the credits across $\{t+1,\ldots,H\}$ are $\{c,c-1,\ldots,1,0,0\}$.  Finally, for $t \ge 3$,  $0$ and $1$ have credits $t-1$ and
$t-2$ respectively, the credits across $\{2,3,\ldots,t-1,t\}$ are $\{t-4,t-5,\ldots,2,1,0,0,0\}$, and the credits across $\{t+1,\ldots,H\}$ are $\{c,c-1,\ldots,1,0,0\}$.  Therefore, noting that
\[
 (t-1)+2(t-2)+2\sum_{i=2}^{t-2}(t-i-2)=(t-2)^2+3,
\]
establishes \eqref{eq:initial-count}.
In particular,
\begin{equation}\label{eq:low-credit}
 A+(t-2)c\ge2\max(t,1)-2.
\end{equation}
For $t\ge2$, the difference between the two sides is
$(t-3)^2+c(H+t-5)/2\ge0$; for $t=0,1$, the first two rows of
\eqref{eq:initial-count} give the inequality directly.

The chosen occurrence of $\ell$ lies outside $\mathcal B$, so its
terms in both sums in \eqref{eq:global} are not counted by $A$ and
sum to $\sigma(\ell)+N(\ell)=t+2-w(\ell)$.
For either other reserved entry $x$, either $A$ counts its credit
(in which case that entry is in $\mathcal B$ and has $N(x)\ge2$), or
else both terms remain uncounted and sum to
$\sigma(x)+N(x)=t+2+c\ind{x>t}-w(x)$.
Each of these two contributions is at least $2$ when $t>0$ and at
least $1$ when $t=0$.
Finally, any entry outside $\mathcal B$ that is not reserved has
nonnegative credit.
Thus, with $r=1$, \eqref{eq:global} and \eqref{eq:low-credit} give
\[
\begin{aligned}
 \Delta&\le1+3(t+2)+(2-t)c
          -\bigl(A+t+2-w(\ell)+2\min(t+1,2)\bigr)\\
       &=1+w(\ell)+2\max(t,1)+(2-t)c-A \le w(\ell)+3=\ell+5\le0.
\end{aligned}
\]

\medskip
\noindent\emph{Case $\ell>-5$.}

First we compute $A:=\sum_{x\in\mathcal B}\sigma(x)$, using as a basis the calculations 
from the preceding case for \eqref{eq:initial-count}.  When $t\le5$, the only nonnegative values whose credit changes are
$\{6,7,8,9\}$ whose credit changes from $\{H-7,H-8,H-9,H-10\}$ to $\{H-8,H-10,H-11,H-11\}$.  When $t \ge 6$ the only nonnegative values whose credit changes are
$\{0,1,2,3\}$ whose credit changes from $\{t-1,t-2,t-4,t-5\}$ to $\{t-2,t-4,t-6,t-6\}$.  
Thus
\[
 A=\begin{cases}
 \binom{c+1}{2}-5,&t=0,\\[2pt]
 \binom{c+1}{2}-4,&t=1,\\[2pt]
 \binom{c+1}{2}+(t-2)^2-3,&2\le t\le5,\\[2pt]
 \binom{c+1}{2}+(t-2)^2-8,&t\ge6.
 \end{cases}
\]
The credit of negative values only increases, remaining positive. 

The sum $N(a)+N(a+1)$ is $10$ and $15$ for $t \le 5$ and $t \ge 6$ respectively.
Together with the formula for $A$, this gives
\begin{equation}\label{eq:high-credit}
 A+N(a)+N(a+1)+(t-3)c
 \ge5t+8+\frac{(H-11)(H+2)}2.
\end{equation}
Indeed, for $t=0,1$, the
difference, LHS minus RHS,  is $15$ and $13$ respectively.
For $t\ge2$, the difference between the two sides is
$(t-6)(t-7)/2\ge0$ (where the nonnegativity follows since $t$ is an integer).

The chosen occurrences of $a$ and $a+1$ lie in $\mathcal B$, so $A$
counts their credits but  not the terms in the last sum
in \eqref{eq:global}. 
For each of the other three reserved entries $x$, either $A$ counts
its credit (in which case that entry is in $\mathcal B$ and has
$N(x)\ge2$), or else both terms remain uncounted and sum to
$\sigma(x)+N(x)=t+2+c\ind{x>t}-w(x)\ge1$.
Each therefore contributes at least $1$ beyond $A$.
Thus, with $r=2$, \eqref{eq:global} and
\eqref{eq:high-credit} give
\[
\begin{aligned}
 \Delta&\le1+5(t+2)+(3-t)c-\bigl(A+N(a)+N(a+1)+3\cdot1\bigr)\le-\frac{(H-11)(H+2)}2\le0.
\end{aligned}
\]
Both subcases contradict $\Delta>0$, proving the proposition.
\end{proof}


\subsection{Reduction on the remaining alphabet}\label{app:remaining-alphabet}
\newtheorem{reductionclaim}{Claim}
We now consider multisets $X$ of size $2k+1$, where $k\ge2$,
$-4\le\ell=\min X\le0$, and $1\le H=\max X\le10$.
Retain $m_i,t,w$ from the preceding proof. The following are immediate from what has already been established.

\begin{reductionclaim}\label{clm:reduction-support}
If $X$ is the content of a simple lower or upper, then
\begin{equation}\label{eq:support}
 m_i\ge1\quad(1\le i\le H),\qquad
 \sum_{i>t}m_i\le k+1-t\quad(t<\infty).
\end{equation}
We also have
\[
\begin{array}{ll}
 t=0: & m_0=0,\\
 1\le t\le H: & m_0\ge1,\quad m_i\ge2\ (1\le i<t),\quad m_t=1,\\
 t=\infty: & m_0\ge1,\quad m_i\ge2\ (1\le i\le H).
\end{array}
\]
\end{reductionclaim}
\begin{proof}
The inequalities in \eqref{eq:support} are the same as those in \eqref{eq:conditions}, whose proof uses
only the simple lower or upper. The restrictions on $t$ and the other
multiplicities follow from the definition of $t$.
\end{proof}

\begin{reductionclaim}\label{clm:reduction-capacity}
Suppose that $X$ is the content of a simple lower or upper and of a
$pTs$ object $\chi=(p,T,s)$ satisfying $2\dinv(\chi)>\beta(X)$. Then
\begin{equation}\label{eq:capacity}
 m_\ell+m_{\ell+1}\le k.
\end{equation}
\end{reductionclaim}
\begin{proof}
First suppose  $\ell<0$. Any occurrence of $\ell$ or $\ell+1$ in $\chi$ must be in the first column of $T$, which has at most $k$ entries. For $\ell=0$, the same is true except that $\ell+1=1$ can also occur in $s$ so $m_0+m_1\le k+1$. We show that equality would force
$2\dinv(\chi)\le\beta(X)$, contradicting our assumption on $\chi$.

Equality simultaneously forces the first column of $T$ to have $k$ entries which are all $0$s and $1$s, $\min(|p|,|s|)=0$, and a $1$ in $s$ (so in fact $|p|=0$).  
Every $2$ then lies in a row $(0,2)$, so $m_2\le m_0$.
Let $g$ be the number of times a $0$ precedes a $1$ or a $1$ precedes a $2$ in the reading word of $\chi$. Since every zero precedes the $1$ in $s$ and every $1$ in the first column contributes one to $g$ for each $(0,2)$ row, we have $g \ge m_0+m_2(m_1-1)$.

Note that the number of $\dinv$ pairs with at least one element equal to $0$ or $1$ is bounded by 
$\binom{k+1}{2}+m_1m_2-g$. On the other hand, the number of $\dinv$ pairs with no element equal to $0$ or $1$ is at most $\binom{k}{2}-f$, where $f$ counts the number of pairs of entries of value at least two and whose values differ by at least $2$. Thus the bound for $g$, $m_0+m_1=k+1$,
and $\beta(X)=2k^2-1+m_1-\sum_{i\ge2}(i-2)m_i$ give

\[
 2\dinv(\chi)-\beta(X)
 \le \left[\sum_{i\ge2}(i-2)m_i+m_2-k-f\right]+(m_2-m_0)-f.
\]
The inequalities $m_i\ge1$ for $1\le i\le H$ in \eqref{eq:support},
together with $\sum_{i\ge2}m_i=k$, give
\[
 f\ge\sum_{i\ge4}(i-3)m_i
   =\sum_{i\ge2}(i-2)m_i+m_2-k.
\]
The bracketed term is therefore nonpositive. Since $m_2\le m_0$
and $f\ge0$, we conclude that $2\dinv(\chi)-\beta(X)\le0$.
This rules out equality and proves \eqref{eq:capacity}.
\end{proof}

For finite words $A,B$, write
\[
 C(A,B)=\#\{(a,b)\in A\times B:a-b\in\{0,1\}\}.
\]
For \emph{dual pairs} (dual Dyck sequences of length $2$) $R,T$, define 
\[
 q(R,T)=4-2\max\{C(R,T),C(T,R)\}.
\]
\begin{reductionclaim}\label{clm:pairwise-loss}
For dual pairs $R,T$, we have $q(R,T)\ge0$,
$q(R,R)=0$, and
\[
 2\dinv(R:T)\le4-q(R,T).
\]
\end{reductionclaim}
\begin{proof}
Each entry of $R$ contributes to at most one pair counted by
$C(R,T)$, since the two entries of $T$ differ by at least two.
Thus $C(R,T)\le2$, proving the first statement. Also $C(R,R)=2$, proving the second, and
$\dinv(R:T)=C(R,T)$, proving the last.
\end{proof}

\paragraph{Factorization frames.}
A \emph{factorization frame} $\mathcal{F}$ for a $pTs$ or $22$-$pTs$ object consists
of a singleton $z$ and two words $p',s'$, with all entries in
$[\ell,H]$, such that $0:p'$ and $s':(-1)$ are affine and reverse
affine, respectively. We require
$|p'|=d=|s'|$, where $d=0$ for a $pTs$ object and $d=2$ for a $22$-$pTs$ object.

Given such a frame $\mathcal{F}$, define a \emph{filling} $F$ of that frame to be a
factorization
\[
 F=R_1:\cdots:R_{k-d},
\]
where each $R_i$ is a dual pair with entries in $[\ell,H]$.
Define the word associated to the frame plus filling to be $p':F:(z):s'$.
Let $e_v$ count the occurrences of $v$ in the fixed entries
$p':(z):s'$, and write $e=(e_v)$. For each possible pair factor $R$,
let $x_R$ count the factors in $F$ equal to $R$. The multiplicities
$m_v$ in the word $p':F:(z):s'$ therefore satisfy
\begin{equation}\label{eq:content}
 m_v=e_v+\sum_{R\ni v}x_R,\qquad
 k=d+\sum_Rx_R,\qquad \sum_ve_v=2d+1.
\end{equation}

For a dual pair $R=(a,b)$, put $W(R)=w(a)+w(b)$.
Define the coefficients
\[
\begin{aligned}
 c&=1+\sum_vw(v)e_v+2\dinv(p':(z):s')-2d^2,\\
 a_R&=W(R)+2C(p',R)+2C(R,(z):s')-4d-2.
\end{aligned}
\]

\begin{reductionclaim}\label{clm:pts-frame}\label{clm:22pts-frame}
Every $pTs$ (resp., $22$-$pTs$) object $\chi=(p,T,s)$ has the same
content and dinv as the word $p':F:(z):s'$ obtained from a frame
$\mathcal{F}$ and filling $F$, where $p'=\emptyset$ (resp., $p'=p_1p_2$) and
$s'=\emptyset$ (resp., $s'=s_{-2}s_{-1}$). Moreover, letting $\prec$
denote the lexicographic ordering on pair factors,
\begin{equation}\label{eq:quadratic}
2\dinv(\chi)-\beta(X)\le c+\sum_Ra_Rx_R
                 -\sum_{R\prec T}q(R,T)x_Rx_T.
\end{equation}
\end{reductionclaim}
\begin{proof}
For a $pTs$ object, the insertion argument in the proof of
Proposition~\ref{thm:main} gives the required frame (the unique singleton) and filling.
For a $22$-$pTs$ object, the same argument applied to the
object disregarding the first two entries of $p$ and the last
two entries of $s$ (which along with the unique singleton of the resulting factorization become the frame) gives the filling (the pair factors in the factorization). 

The constructed filling, $F$, consists of $k-d$ dual pairs, so
Claim~\ref{clm:pairwise-loss} gives
\[
 2\dinv(F)\le 2(k-d)(k-d-1)
 -\sum_{R\prec T}q(R,T)x_Rx_T.
\]
Using $\beta(X)=2k^2-1-\sum_vw(v)m_v$ and \eqref{eq:content}, we obtain
\[
\begin{aligned}
2\dinv(\chi)-\beta(X)
={}&2\dinv(p':(z):s')+2\dinv(F)+2C(p',F)+2C(F,(z):s')\\
 &+\sum_vw(v)m_v-2k^2+1\\
\le{}&1+\sum_vw(v)e_v+2\dinv(p':(z):s')\\
 &+\sum_R\left(W(R)+2C(p',R)+2C(R,(z):s')\right)x_R\\
 &+2(k-d)(k-d-1)-2k^2-\sum_{R\prec T}q(R,T)x_Rx_T\\
={}&c+\sum_Ra_Rx_R-\sum_{R\prec T}q(R,T)x_Rx_T.
\end{aligned}
\]
The last equality uses $k-d=\sum_Rx_R$ and
\[
 2(k-d)(k-d-1)-2k^2=-(4d+2)(k-d)-2d^2.
\]
\end{proof}

\paragraph{Original-row frames.}
An \emph{original-row frame} $\mathcal{F}$ for a $pTs$ or $22$-$pTs$ object consists of
the outer sequences $p,s$ and the word $J$ listing the entries in
the single-box rows in weakly decreasing order. Put $d=\max\{|p|,|s|\}-1$, let $e_v$ count the
occurrences of $v$ in $p:J:s$, and let $x_R$ count the two-box rows
of type $R$. Set
\[
\begin{aligned}
 c&=1+\sum_vw(v)e_v+2\dinv(p:J:s)-2d^2,\\
 a_R&=W(R)+2C(p:J,R)+2C(R,s)-4d-2.
\end{aligned}
\]

\begin{reductionclaim}[Original tableau rows]\label{clm:original-row-frame}
Every $pTs$ or $22$-$pTs$ object $\chi=(p,T,s)$ has an original-row
frame $\mathcal{F}$ and filling $F$, where $F$ is the word of the
two-box rows of $T$ in reading order. The associated word $p:J:F:s$
has the same content as $\chi$ and satisfies
$\dinv(\chi)\le\dinv(p:J:F:s)$. Moreover, with the coefficients
defined above,
\[
 2\dinv(\chi)-\beta(X)\le c+\sum_Ra_Rx_R
 -\sum_{R\prec T}q(R,T)x_Rx_T.
\]
For a fixed alphabet $[\ell,H]$, there are only finitely many such
frames.
\end{reductionclaim}
\begin{proof}
Let $F=R_1:\cdots:R_{k-d}$ be the word obtained by reading the
two-box rows from bottom to top. The single-box rows occur at the
bottom of the tableau, so their entries form one block between $p$
and $F$ in the reading word. Ordering this block decreasingly
maximizes its internal dinv and leaves its interactions with the
surrounding blocks unchanged. The resulting word $p:J:F:s$
therefore has dinv at least $\dinv(\chi)$.

Claim~\ref{clm:pairwise-loss} gives
\[
 2\dinv(F)\le 2(k-d)(k-d-1)
 -\sum_{R\prec T}q(R,T)x_Rx_T.
\]
Using $\beta(X)=2k^2-1-\sum_vw(v)m_v$, we obtain
\[
\begin{aligned}
2\dinv(\chi)-\beta(X)
\le{}&2\dinv(p:J:s)+2\dinv(F)+2C(p:J,F)+2C(F,s)\\
 &+\sum_vw(v)m_v-2k^2+1\\
\le{}&1+\sum_vw(v)e_v+2\dinv(p:J:s)\\
 &+\sum_R\left(W(R)+2C(p:J,R)+2C(R,s)\right)x_R\\
 &+2(k-d)(k-d-1)-2k^2-\sum_{R\prec T}q(R,T)x_Rx_T\\
={}&c+\sum_Ra_Rx_R-\sum_{R\prec T}q(R,T)x_Rx_T.
\end{aligned}
\]
The last equality uses $k-d=\sum_Rx_R$ and
\[
 2(k-d)(k-d-1)-2k^2=-(4d+2)(k-d)-2d^2.
\]

The alphabet bounds the lengths of $p,s$: after their initial affine
parts, entries increase by at least two in the relevant direction.
The size of $J$ is then bounded as well. Hence there are finitely many possibilities for $p,s,J$.
\end{proof}

\paragraph{Matrices associated with a frame.}
Fix $\ell,H,t$ and a frame $\mathcal{F}$. The frame determines $d,e,c,(a_R)$.
Define the incidence matrix $A$, with rows indexed by values
$v\in[\ell,H]$ and columns indexed by possible dual pairs $R$, by
\[
 A_{vR}=\mathbf1_{\{v\in R\}}.
\]
Let $L_v$ denote the multiplicity of $v$ in the multiset union of
$\{1,\ldots,H\}$ and $\{0,\ldots,\min(t-1,H)\}$, with one copy of
$\ell$ adjoined if it is not already present, and write $L=(L_v)$.

Next define a matrix $B$, again with columns indexed by possible
dual pairs, and a vector $b$ by the following rows:
\begin{enumerate}
\item Row $1$ of $Bx\le b$ represents $m_\ell+m_{\ell+1}\le k$:
\[
 B_{1R}=A_{\ell R}+A_{\ell+1,R}-1,\qquad
 b_1=d-e_\ell-e_{\ell+1}.
\]
\item Row $2$ of $Bx\le b$ represents $k\ge2$:
\[
 B_{2R}=-1,\qquad b_2=d-2.
\]
\item When $t<\infty$, row $3$ of $Bx\le b$ represents
$\sum_{v>t}m_v\le k+1-t$:
\[
 B_{3R}=\sum_{v>t}A_{vR}-1,\qquad
 b_3=d+1-t-\sum_{v>t}e_v.
\]
\item When $t<\infty$, row $4$ of $Bx\le b$ represents
$m_0\le0$ if $t=0$, and $m_t\le1$ if $1\le t\le H$:
\[
 B_{4R}=A_{tR},\qquad
 b_4=\begin{cases}
 -e_0,&t=0,\\
 1-e_t,&1\le t\le H.
 \end{cases}
\]
\end{enumerate}
The matrix $A$ has $H-\ell+1\le15$ rows, and $B$ has four rows for
finite $t$ and two for $t=\infty$. Both have at most $91$ columns.
With $\ell,H,t$ fixed, $A$ and $B$ are the same for every frame;
$b$ depends on $d,e$.

\begin{reductionclaim}\label{clm:structural}
Suppose $X$ has the fixed parameters $\ell,H,t$ and is the content
of a simple lower or upper, a $pTs$ object, and a $22$-$pTs$ object,
each with doubled dinv exceeding $\beta(X)$. For each tableau object,
use either Claim~\ref{clm:pts-frame} or
Claim~\ref{clm:original-row-frame} to choose a frame together with
a filling: $\mathcal{F},F$ for the $pTs$ object and
$\mathcal{F}',F'$ for the $22$-$pTs$ object. Let $x,x'$ be the pair
multiplicity vectors of $F,F'$, respectively, and use primes for
the constants associated with $\mathcal{F}'$. Then
\begin{equation}\label{eq:structural}
\begin{aligned}
 Ax&\ge L-e, & Bx&\le b,\\
 Ax'&\ge L-e', & Bx'&\le b',\\
 A(x-x')&=e'-e, & x,x'&\ge0.
\end{aligned}
\end{equation}
\end{reductionclaim}
\begin{proof}
The inequalities involving $L$ follow from
Claim~\ref{clm:reduction-support} and the definitions of $t$ and
$\ell$.  The inequalities for the rows of $B$ follow, respectively, from 
Claim~\ref{clm:reduction-capacity}, the assumption $k\ge2$,
the second inequality in \eqref{eq:support}, and the definition of $t$.
The content identities give $A(x-x')=e'-e$, and the pair
multiplicities $x,x'$ are nonnegative.
\end{proof}

\subsection{Linearization}\label{app:linearization}
\paragraph{Linear constraints for products.} We now introduce
nonnegative variables,  $Z_{RT}=Z_{TR}$, $Z'_{RT}=Z'_{TR}$ and  $Y_{RT}$.
We obtain  a number of helpful linear constraints by multiplying the relations of \eqref{eq:structural}
and using \eqref{eq:quadratic}
 and replacing $x_Rx_T$, $x'_Rx'_T$, and $x_Rx'_T$ by
$Z_{RT}$, $Z'_{RT}$, and $Y_{RT}$, respectively. 
Precisely, fix a $pTs$ frame $\mathcal{F}$ and a $22$-$pTs$ frame $\mathcal{F}'$, each of which
may be a factorization frame or an original-row frame. Consider the linear system consisting of \eqref{eq:structural} and the constraints \eqref{eq:product-support} derived below.

\begin{subequations}\label{eq:product-support}
\begin{enumerate}
\item Multiply $A_vx\ge L_v-e_v$ and $B_jx\le b_j$ by each
$x_T\ge0$, and their primed counterparts by each $x'_T\ge0$, to obtain
\begin{equation}\label{eq:count-products}
\begin{aligned}
 \sum_R A_{vR}Z_{RT}&\ge(L_v-e_v)x_T,
 &\sum_R B_{jR}Z_{RT}&\le b_jx_T,\\
 \sum_R A_{vR}Z'_{RT}&\ge(L_v-e'_v)x'_T,
 &\sum_R B_{jR}Z'_{RT}&\le b'_jx'_T.
\end{aligned}
\end{equation}
\item Multiply $A_vx\ge L_v-e_v$ and $B_jx\le b_j$ by $x'_T\ge0$ to get
\begin{equation}\label{eq:mixed-count-products}
 \sum_R A_{vR}Y_{RT}\ge(L_v-e_v)x'_T,\qquad
 \sum_R B_{jR}Y_{RT}\le b_jx'_T.
\end{equation}
\item Multiply any two of the nonnegative quantities
$A_vx-L_v+e_v$ and $b_j-B_jx$, including a quantity with itself,
and repeat with their primed counterparts. 
\begin{equation}\label{eq:row-products}
\begin{aligned}
 \sum_{R,T}A_{vR}A_{uT}Z_{RT}
 &\ge (L_v-e_v)A_ux+(L_u-e_u)A_vx-(L_v-e_v)(L_u-e_u),\\
 \sum_{R,T}A_{vR}B_{jT}Z_{RT}
 &\le b_jA_vx+(L_v-e_v)B_jx-(L_v-e_v)b_j,\\
 \sum_{R,T}B_{iR}B_{jT}Z_{RT}
 &\ge b_iB_jx+b_jB_ix-b_ib_j,\\[0.5ex]
 \sum_{R,T}A_{vR}A_{uT}Z'_{RT}
 &\ge (L_v-e'_v)A_ux'+(L_u-e'_u)A_vx'-(L_v-e'_v)(L_u-e'_u),\\
 \sum_{R,T}A_{vR}B_{jT}Z'_{RT}
 &\le b'_jA_vx'+(L_v-e'_v)B_jx'-(L_v-e'_v)b'_j,\\
 \sum_{R,T}B_{iR}B_{jT}Z'_{RT}
 &\ge b'_iB_jx'+b'_jB_ix'-b'_ib'_j.
\end{aligned}
\end{equation}
\item Multiply each content-matching equation
$A_v(x-x')=e'_v-e_v$ by $x_T$ and $x'_T$, respectively, to obtain
\begin{equation}\label{eq:matching-products}
\begin{aligned}
 \sum_R A_{vR}(Z_{RT}-Y_{TR})&=(e'_v-e_v)x_T,\\
 \sum_R A_{vR}(Y_{RT}-Z'_{RT})&=(e'_v-e_v)x'_T.
\end{aligned}
\end{equation}
\item Finally, impose the linearized dinv bounds from
Claims~\ref{clm:pts-frame} and~\ref{clm:original-row-frame}:
\begin{equation}\label{eq:linear}
\begin{aligned}
 c+\sum_Ra_Rx_R-
 \sum_{R\prec T}q(R,T)Z_{RT}&\ge1,\\
 c'+\sum_Ra'_Rx'_R-
 \sum_{R\prec T}q(R,T)Z'_{RT}&\ge1.
\end{aligned}
\end{equation}
\end{enumerate}
\end{subequations}

\begin{proposition}\label{prop:reduction}
Fix $\ell,H,t$, where $-4\le\ell\le0$,
$1\le H\le10$, and $t\in\{0,1,\ldots,H,\infty\}$.
Let $X$ be a multiset with these parameters and $|X|=2k+1$,
where $k\ge2$. If there exist a simple lower or upper,
a $pTs$ object, and a $22$-$pTs$ object, each with content $X$ and
doubled dinv exceeding $\beta(X)$, then, for every choice of
factorization or original-row type for each tableau class, there
exist frames $\mathcal{F},\mathcal{F}'$ of those types whose
associated system \eqref{eq:structural}, \eqref{eq:product-support}
has a nonnegative real solution.
\end{proposition}
\begin{proof}
Fix the frame type for each tableau class. Apply the corresponding
construction of Claim~\ref{clm:pts-frame} or
Claim~\ref{clm:original-row-frame} to the two tableau objects,
obtaining frames $\mathcal{F},\mathcal{F}'$ with fillings $F,F'$.
Let $x,x'$ be the pair multiplicity vectors of $F,F'$, respectively.
Claim~\ref{clm:structural} gives \eqref{eq:structural}. Set
\[
 Z_{RT}=x_Rx_T,\qquad Z'_{RT}=x'_Rx'_T,\qquad Y_{RT}=x_Rx'_T.
\]
With these assignments, \eqref{eq:count-products}--\eqref{eq:matching-products}
follow from \eqref{eq:structural}. Claims~\ref{clm:pts-frame}
and~\ref{clm:original-row-frame} give \eqref{eq:linear}, since
each object's doubled dinv exceeds $\beta(X)$ by at least $1$.
\end{proof}

\subsection{Exact computer checks}\label{app:exact-checks}
\begin{reductionclaim}\label{clm:checked-infeasibility}
Fix $\ell,H,t$, where  $-4\le\ell\le0$,
$1\le H\le10$, $(\ell,H)\ne(0,1)$, and
$t\in\{0,1,\ldots,H,\infty\}$. There is a choice of frame type
for each tableau class such that, for every pair of frames of
those types with fixed entries in $[\ell,H]$, the system
\eqref{eq:structural}, \eqref{eq:product-support} has no
nonnegative real solution.
\end{reductionclaim}
\begin{proof}
The following choices of frame types depend only on $\ell,H$
and work for every $t$.
\begin{center}
\begin{tabular}{@{}p{.52\linewidth}p{.21\linewidth}p{.21\linewidth}@{}}
Endpoint ranges & $pTs$ frame & $22$-$pTs$ frame\\ \hline
$\ell \in [-4,-3]$ & Factorization & Factorization\\
$\ell=-2$, $H\in[1,3]$ & Original rows & Original rows\\
$\ell=-2$, $H\in[4,10]$ & Factorization & Factorization\\
$\ell=-1$, $H\in[1,5]$ & Original rows & Original rows\\
$\ell=-1$, $H\in[6,10]$ & Original rows & Factorization\\
$\ell=0$, $H\in[2,6]$ & Original rows & Original rows\\
$\ell=0$, $H\in[7,10]$ & Original rows & Factorization
\end{tabular}
\end{center}
The function \texttt{verify()} in the code below checks these
choices for every $t$ and returns \texttt{True}, with each
exclusion verified in exact rational arithmetic. 
\end{proof}

\paragraph{Implementation.}
The code runs through the values of $\ell,H,t$ in
Claim~\ref{clm:checked-infeasibility} and uses the table above to
choose the frame types. It enumerates the corresponding frames
with fixed entries in $[\ell,H]$ and builds subsets of
\eqref{eq:structural} and \eqref{eq:product-support}.
When constraints from one frame alone prove infeasibility, the
code rules out every pair containing that frame at once. For
other pairs, it combines constraints from both frames, including
the equations matching their contents. 

Numerical linear programming searches for infeasibility
certificates, which the code checks in exact rational arithmetic.
Each accepted certificate combines the system inequalities to
give a contradiction. The checker returns \texttt{True} only when
every required exclusion has been certified; otherwise it reports
an uncertified case.

To run the check, save the listing as \texttt{appendix\_checks.py}.
Using Python 3.10 or later, install \texttt{numpy}, \texttt{scipy},
and \texttt{sympy}, and run \texttt{python appendix\_checks.py}.

\begingroup
\scriptsize
\makeatletter
\begingroup
\catcode`'=\active
\catcode``=\active
\g@addto@macro\@noligs{\let'\textquotesingle\let`\textasciigrave}
\endgroup
\makeatother
\begin{verbatim}
#!/usr/bin/env python3
"""Root LP checks with rational certificates; Python 3.10+.
Install numpy, scipy, sympy; run python appendix_checks.py.
"""
from collections import Counter
from itertools import combinations, combinations_with_replacement, product
from fractions import Fraction

def w(v):
    return v + 2 if v < 0 else (0 if v == 0 else v - 2)

def W(R):
    return sum(w(v) for v in R)

def C(A, B):
    return sum(a - b in (0, 1) for a in A for b in B)

def dinv(word):
    return sum(a - b in (0, 1) for i, a in enumerate(word) for b in word[i + 1:])

class Data:
    def __init__(self, ell, H, t):
        self.ell, self.H, self.t = ell, H, t   # None represents t = infinity.
        self.values = tuple(range(ell, H + 1))
        self.pairs = tuple((a, b) for a, b in combinations(self.values, 2)
                           if b - a >= 2)
        self.n_pairs = len(self.pairs)
        # Value indices are v - ell; pair indices follow self.pairs.
        self.A = [[int(v in R) for R in self.pairs] for v in self.values]
        self.q = [[2 * (2 - max(C(R, T), C(T, R)))
                       for T in self.pairs] for R in self.pairs]
        self.L = [int(v == ell or v >= 1) for v in self.values]
        if t != 0:
            self.L[-ell] = 1
            for v in range(1, H + 1 if t is None else t):
                self.L[v - ell] = 2
        self.B = [[self.A[0][i] + self.A[1][i] - 1
                   for i in range(self.n_pairs)], [-1] * self.n_pairs]
        if t is not None:
            self.B.append([sum(self.A[v - ell][i] for v in self.values if v > t) - 1
                           for i in range(self.n_pairs)])
            self.B.append(self.A[t - ell][:])

    def content(self, word):
        return tuple(word.count(v) for v in self.values)

    def generic00(self, z):
        return dict(d=0, e=self.content((z,)), c=1 + w(z), z=z,
                    a=tuple(W(R) + 2 * C(R, (z,)) - 2 for R in self.pairs))

    def generic22(self, p_prime, s_prime, z):
        fixed_word = p_prime + (z,) + s_prime
        return dict(d=2, e=self.content(fixed_word),
                    c=1 + sum(map(w, fixed_word)) + 2 * dinv(fixed_word) - 8,
                    z=z, p_prime=p_prime, s_prime=s_prime,
                    a=tuple(W(R) + 2 * C(p_prime, R)
                            + 2 * C(R, (z,) + s_prime) - 10 for R in self.pairs))

    def generic22_roots(self):
        prefixes = [p_prime for p_prime in product(self.values, repeat=2)
                    if p_prime[0] <= 1 and p_prime[1] <= p_prime[0] + 1]
        suffixes = [s_prime for s_prime in product(self.values, repeat=2)
                    if s_prime[1] <= 0 and s_prime[0] <= s_prime[1] + 1]
        frames = []
        for p_prime, s_prime, z in product(prefixes, suffixes, self.values):
            frame_prime = self.generic22(p_prime, s_prime, z)
            frames.append(frame_prime)
        return frames

    def original_frames(self, kind):
        def outer(prefix):
            answer = []
            def visit(word):
                length = len(word)
                if kind == 0 or length >= 2:
                    answer.append(word if prefix else word[::-1])
                previous = word[-1] if word else (0 if prefix else -1)
                for v in self.values:
                    admissible = (v <= previous + 1 if kind == 2 and length < 2
                                  else v >= previous + 2)
                    if admissible:
                        visit(word + (v,))
            visit(())
            return answer
        frames = []
        for p in outer(True):
            for s in outer(False):
                len_p, len_s = len(p), len(s)
                if len_p == len_s:
                    continue
                d = max(len_p, len_s) - 1
                for multiset in combinations_with_replacement(
                        self.values, abs(len_p - len_s) - 1):
                    J = tuple(reversed(multiset))
                    pJ = p + J
                    fixed_word = pJ + s
                    e = self.content(fixed_word)
                    c = 1 + sum(map(w, fixed_word)) + 2 * dinv(fixed_word) - 2*d*d
                    a = tuple(W(R) + 2*C(pJ, R) + 2*C(R, s) - 4*d - 2
                              for R in self.pairs)
                    frame = dict(d=d, e=e, c=c, a=a)
                    frames.append(frame)
        return frames

    def structural_system(self, frame, frame_prime=None):
        """A single frame is arbitrary; paired frames are pTs and 22-pTs."""
        n_pairs = self.n_pairs
        zeros = [0] * n_pairs
        extend = (lambda row: row + zeros) if frame_prime is not None else list
        d, e = frame['d'], frame['e']
        b = [d - e[0] - e[1], d - 2]
        if self.t is not None:
            b.append(d + 1 - self.t - sum(e[v - self.ell]
                                         for v in self.values if v > self.t))
            b.append((1 if self.t > 0 else 0) - e[self.t - self.ell])
        # Stack -Ax <= e-L and Bx <= b; keep content matching separate.
        ineq_rows = [extend([-coeff for coeff in row]) for row in self.A]
        ineq_rows.extend(extend(row) for row in self.B)
        ineq_rhs = [e[i] - self.L[i] for i in range(len(self.values))] + b
        match_rows, match_rhs = [], []
        if frame_prime is not None:
            for i, row in enumerate(self.A):
                match_rows.append(row[:] + [-coeff for coeff in row])
                match_rhs.append(frame_prime['e'][i] - e[i])
        n_counts = n_pairs if frame_prime is None else 2 * n_pairs
        return n_counts, ineq_rows, ineq_rhs, match_rows, match_rhs

    def system(self, frame, frame_prime=None, stage=2):
        """Stage 0: support products; 1: all count products; 2: all products."""
        n_counts, ineq_rows, ineq_rhs, match_rows, match_rhs = (
            self.structural_system(frame, frame_prime))
        frames = [frame] if frame_prime is None else [frame, frame_prime]
        # In paired systems, count columns are x_R followed by x'_R.
        # Their product columns represent Z_RT, Y_RT, and Z'_RT.
        product_columns = {}
        for i in range(n_counts):
            for j in range(i, n_counts):
                if stage or (i//self.n_pairs == j//self.n_pairs and
                             self.q[i % self.n_pairs][j % self.n_pairs]):
                    product_columns[i, j] = n_counts + len(product_columns)
        n_variables = n_counts + len(product_columns)
        def product_column(i, j):
            return product_columns[min(i, j), max(i, j)]
        rows = [{i:coeff for i,coeff in enumerate(row) if coeff} for row in ineq_rows]
        rhs = list(ineq_rhs)
        for side, current_frame in enumerate(frames):
            offset = side*self.n_pairs
            row = {offset+i:-a_R for i,a_R in enumerate(current_frame['a']) if a_R}
            for i in range(self.n_pairs):
                for j in range(i+1, self.n_pairs):
                    if self.q[i][j]:
                        row[product_column(offset+i, offset+j)] = self.q[i][j]
            rows.append(row); rhs.append(current_frame['c']-1)
        for coefficients, value in zip(match_rows, match_rhs):
            row = {i:coeff for i,coeff in enumerate(coefficients) if coeff}
            rows.extend([row, {i:-coeff for i,coeff in row.items()}])
            rhs.extend([value, -value])
        if stage == 0:
            for side, current_frame in enumerate(frames):
                offset = side*self.n_pairs
                for i in range(self.n_pairs):
                    for value_index, A_v in enumerate(self.A):
                        need = self.L[value_index] - current_frame['e'][value_index]
                        incident = [j for j,coeff in enumerate(A_v) if coeff]
                        if need <= 0 or any(self.q[i][j] == 0 for j in incident):
                            continue
                        row = {product_column(offset+i, offset+j):-1 for j in incident}
                        row[offset+i] = need
                        rows.append(row); rhs.append(0)
            return n_variables, rows, rhs
        # Count times inequality, or equality in both directions.
        equations = list(zip(match_rows, match_rhs))
        constraints = list(zip(ineq_rows, ineq_rhs)) + equations + [
            ([-coeff for coeff in row], -value) for row,value in equations]
        for coefficients, value in constraints:
            for i in range(n_counts):
                row = {product_column(i,j):coeff
                       for j,coeff in enumerate(coefficients) if coeff}
                if value:
                    row[i] = -value
                rows.append(row); rhs.append(0)
        if stage == 2:
            for pos,(row_i,bound_i) in enumerate(zip(ineq_rows, ineq_rhs)):
                for row_j,bound_j in list(zip(ineq_rows, ineq_rhs))[pos:]:
                    row = {i:bound_j*coeff_i + bound_i*coeff_j
                           for i,(coeff_i,coeff_j) in enumerate(zip(row_i,row_j))
                           if bound_j*coeff_i + bound_i*coeff_j}
                    for i,coeff_i in enumerate(row_i):
                        if coeff_i:
                            for j,coeff_j in enumerate(row_j):
                                if coeff_j:
                                    column = product_column(i,j)
                                    row[column] = row.get(column,0) - coeff_i*coeff_j
                    rows.append({i:coeff for i,coeff in row.items() if coeff})
                    rhs.append(bound_i*bound_j)
        return n_variables, rows, rhs

STATS = Counter()

def valid(rows, rhs, lam):
    """Check a Farkas certificate over the rationals, not in floating point."""
    if lam is None or len(lam) != len(rows) or any(value < 0 for value in lam):
        return False
    if sum(value*bound for value,bound in zip(lam, rhs)) >= 0:
        return False
    sums = {}
    for value, row in zip(lam, rows):
        if value:
            for j,coeff in row:
                sums[j] = sums.get(j, 0) + value*coeff
    return all(value >= 0 for value in sums.values())

def reconstruct(rows, rhs, numerical):
    """Exact active-set proposal; valid() must still check its signs."""
    from sympy.polys.matrices import DomainMatrix
    from sympy import ZZ
    support = [i for i,value in enumerate(numerical) if value > 1e-10]
    if not support:
        return None
    trans = {}
    for local,i in enumerate(support):
        for j,coeff in rows[i]:
            trans.setdefault(j, {})[local] = coeff
    vals = [float(numerical[i]) for i in support]
    equations = [{j:rhs[i] for j,i in enumerate(support) if rhs[i]}]
    for col in trans.values():
        if abs(sum(coeff*vals[i] for i,coeff in col.items())) < 1e-7:
            equations.append(col)
    certificate_matrix = DomainMatrix(
        {i:{j:ZZ(coeff) for j,coeff in row.items()} for i,row in enumerate(equations)},
        (len(equations), len(support)), ZZ)
    certificate_rhs = DomainMatrix({0:{0:ZZ(-1)}}, (len(equations), 1), ZZ)
    try:
        nums, den = certificate_matrix.solve_den(certificate_rhs, method='rref')
    except Exception:
        return None
    lam = [Fraction(0)] * len(rows)
    for i,value in zip(support, nums.to_list()):
        lam[i] = Fraction(int(value[0]), int(den))
    return lam

def propose(n_variables, rows, rhs):
    import numpy as np
    from scipy.optimize import linprog
    from scipy.sparse import csr_matrix
    row_indices, col_indices, coefficients = [], [], []
    for i,row in enumerate(rows):
        for j,coeff in row:
            row_indices.append(i); col_indices.append(j); coefficients.append(coeff)
        row_indices.append(i); col_indices.append(n_variables); coefficients.append(-1)
    constraint_matrix = csr_matrix(
        (np.asarray(coefficients, float), (row_indices, col_indices)),
        shape=(len(rows), n_variables+1))
    objective = np.zeros(n_variables+1); objective[-1] = 1
    result = linprog(objective, A_ub=constraint_matrix, b_ub=np.asarray(rhs, float),
                     bounds=(0, None), method='highs', options={'time_limit':60.0})
    if not result.success or result.fun <= 1e-8:
        return None
    numerical = -result.ineqlin.marginals
    for den in (100, 10000, 1000000, 1000000000):
        lam = [Fraction(float(value)).limit_denominator(den) for value in numerical]
        if valid(rows, rhs, lam):
            return lam
    lam = reconstruct(rows, rhs, numerical)
    if valid(rows, rhs, lam):
        STATS['reconstructed'] += 1
        return lam
    return None

def certified(system):
    """True means proved infeasible; False means only not certified."""
    STATS['LP_calls'] += 1
    n_variables, rows, rhs = system
    clean = {}
    for row,bound in zip(rows, rhs):
        key = tuple(sorted((j,coeff) for j,coeff in row.items() if coeff))
        if not key:
            if bound < 0:
                STATS['constant_certificates'] += 1
                return True
        else:
            clean[key] = min(bound, clean.get(key, bound))
    if not clean:
        return False
    rows, rhs = list(clean), list(clean.values())
    lam = propose(n_variables, rows, rhs)
    if valid(rows, rhs, lam):
        STATS['rational_certificates'] += 1
        return True
    return False

def exclude(D, frame, frame_prime=None):
    STATS['joint_roots' if frame_prime is not None else 'single_roots'] += 1
    for stage in range(3):
        system = D.system(frame, frame_prime, stage)
        if certified(system):
            STATS['closed_stage_' + str(stage)] += 1
            return
    raise ArithmeticError(('Uncertified root', D.ell, D.H, D.t, frame, frame_prime))

def survivors(D, kind):
    frames = D.original_frames(kind)
    STATS['screened_frames'] += len(frames)
    return [frame for frame in frames if not certified(D.system(frame, stage=0))]

def verify_case(ell, H):
    if (ell, H) == (0, 1):
        return True                 # Excluded by capacity.
    for t in list(range(H + 1)) + [None]:
        D = Data(ell, H, t)
        if ell <= -3 or ell == -2 and H >= 4:
            for z in D.values:
                exclude(D, D.generic00(z))
        elif ell == -1 and H >= 6 or ell == 0 and H >= 7:
            pts_frames = None
            for frame_prime in D.generic22_roots():
                exceptional = (ell == -1 and H == 6 and t == 2
                               and frame_prime['p_prime'] == (1, 1)
                               and frame_prime['s_prime'] in ((-1, 0), (0, -1))
                               and frame_prime['z'] in (0, 1))
                if exceptional:
                    if pts_frames is None:
                        pts_frames = survivors(D, 0)
                    for frame in pts_frames:
                        exclude(D, frame, frame_prime)
                else:
                    exclude(D, frame_prime)
        else:
            pts_frames = survivors(D, 0)
            if not pts_frames:
                continue
            pts22_frames = survivors(D, 2)
            for frame,frame_prime in product(pts_frames, pts22_frames):
                exclude(D, frame, frame_prime)
    return True

def verify():
    return all(verify_case(ell,H) for ell in range(-4,1) for H in range(1,11))

if __name__ == '__main__':
    print(verify())
\end{verbatim}
\endgroup

\subsection{Proof of Lemma~\ref{ineq}}\label{app:completion}
\begin{proof}
If any of the three classes is empty, the result is immediate. Otherwise,
$|X|=2k+1$ with $k\ge2$, since no $22$-$pTs$ object exists for
$k<2$. Put $\ell=\min X$ and $H=\max X$.
The endpoint conditions force $\ell\le0$ and $H\ge1$: if $\ell>0$,
there is no simple lower or upper, and if $H<1$, there is no
$pTs$ object.

If $\ell\le-5$ or $H\ge11$, Proposition~\ref{thm:main} gives
$2\min\{h_z(X),h_0(X)\}\le\beta(X)$, which implies the required
bound. Otherwise, $-4\le\ell\le0$ and $1\le H\le10$.
Suppose \eqref{eq:nonexistence-bound} fails, so that each of the
three classes contains an object with doubled dinv exceeding
$\beta(X)$. If $(\ell,H)=(0,1)$, then $m_0+m_1=2k+1>k$,
contradicting Claim~\ref{clm:reduction-capacity}. Otherwise,
Claim~\ref{clm:checked-infeasibility} gives frame types for which
every associated system is infeasible, whereas
Proposition~\ref{prop:reduction} gives a nonnegative real solution
using frames of those types from the two tableau objects.
This contradiction proves the lemma.
\end{proof}

\end{document}